\documentclass[12pt,reqno]{amsart}
\usepackage{amsmath,amsthm,amstext}
\usepackage{amsfonts,amssymb}
\usepackage[mathscr]{eucal}
\usepackage[]{hyperref}
\usepackage{setspace,url}
\usepackage{tikz}
\usepackage{bm}
\usepackage{subfig}
\usepackage{graphicx}
\usepackage{enumerate}

\newtheorem{lemma}{Lemma}

\newtheorem{proposition}{Proposition}
\newtheorem{remark}{Remark}

\numberwithin{equation}{section}

\numberwithin{equation}{section}

\renewcommand{\zeta}{z}
\renewcommand{\epsilon}{\varepsilon}
\renewcommand{\bar}{\overline}

\newcommand{\rk}{{\rm k}}

\newcommand{\reals}{\mathbb{R}}

\newcommand{\RR}{\ensuremath{\mathbb{R}}}

\newcommand{\bbR}{\ensuremath{\mathbb{R}}} 
\newcommand{\calO}{\ensuremath{\mathcal{O}}} 

\newcommand{\R}{\mathbb{R}}

\newcommand{\bx}{{\rm x}}
\newcommand{\bX}{{\rm X}}
\newcommand{\bn}{{\rm n}}

\newcommand{\K}{{\kappa}}
\newcommand{\sgn}{{\rm sgn}}
\newcommand{\etan}{\eta'}
\newcommand{\xin}{\xi'}
\newcommand{\etah}{\widetilde{\eta}}
\newcommand{\xih}{\widetilde{\xi}}

\title[Hamiltonian Dysther equation for 3D deep water waves]{HAMILTONIAN DYSTHE EQUATION FOR 3D DEEP-WATER
GRAVITY WAVES}

\author[P. Guyenne]{Philippe Guyenne}
\address[P. Guyenne]{Department of Mathematical Sciences, University of Delaware, Newark, DE 19716}
\email{guyenne@udel.edu}

\author[A. Kairzhan]{Adilbek Kairzhan}
\address[A. Kairzhan]{Department of  Mathematics, University of Toronto, Ontario, M5S2E4, Canada}
\email{kairzhan@math.toronto.edu}

\author[C. Sulem]{Catherine Sulem}
\address[C. Sulem]{Department of  Mathematics, University of Toronto, Ontario, M5S2E4, Canada}
\email{sulem@math.utoronto.ca}

\newcommand\blfootnote[1]{%
  \begingroup
  \renewcommand\thefootnote{}\footnote{#1}%
  \addtocounter{footnote}{-1}%
  \endgroup
}

\begin{document}

\maketitle

\begin{abstract}
This article concerns the water wave problem in a three-dimensional domain of infinite depth and examines the modulational regime for weakly nonlinear wavetrains. We use the method of normal form transformations near the equilibrium state to provide a new derivation of the Hamiltonian Dysthe equation describing the slow evolution of the wave envelope. A precise calculation of the third-order normal form allows for a refined reconstruction of the free surface. We test our approximation against direct numerical simulations of the three-dimensional Euler system and against predictions from the classical Dysthe equation, and find very good agreement. 
\end{abstract}

\blfootnote{\textit{Keywords}: 3D deep-water waves,  Hamiltonian systems,  Normal form transformations,  Modulational Analysis,  Dysthe equation, Numerical simulations \\
\textit{2020 Mathematics Subject Classification:} 76B15, 35Q55}



\section{Introduction}

Modulation theory is a well-established theory  to study the long-time evolution and stability of oscillatory
solutions to partial differential equations. In the setting of a modulational regime,
an Ansatz for the solutions is introduced   in the form of  a weakly nonlinear modulated wavetrain and one
 derives reduced equations describing the evolution of its slowly varying envelope.
In the context of surface gravity waves, one finds the Nonlinear Schr\"odinger (NLS) equation, or more generally
the Davey--Stewartson system in three  dimensions. A higher-order approximation was proposed by Dysthe \cite{D79}  for deep water, using 
the perturbative method of  multiple  scales.
 It was  later extended to other settings such as   finite depth \cite{BJ86},
 gravity-capillary waves \cite{H85},  exact linear dispersion \cite{TKDV00}, waves in the presence of  dissipation \cite{HM91},   
 even to higher order \cite{SP20}.  The Dysthe equation  and its variants have been widely used in the water wave community due to their efficiency at describing realistic waves, 
in particular waves with moderately large steepness. Such a model exhibits contributions from the mean flow induced by radiation stresses of the modulated wavetrain, which in turn leads to  improvement in the stability properties of finite-amplitude waves.

However, unlike the NLS equation,  earlier versions of the  Dysthe equation are not Hamiltonian while the original water wave system has  a  Hamiltonian structure \cite{Z68}. 
Gramstad and Trulsen \cite{GT11}  used a refined version of Zakharov's four-wave  interaction model
as obtained by Krasitskii \cite{K94} and derived a Hamiltonian version of Dysthe's equation for  three-dimensional gravity surface
 waves on finite depth.  
 Craig et al. \cite{CGS21} considered the two-dimensional problem of gravity waves on deep water and derived a Hamiltonian Dysthe
 equation from the original water wave system  through a sequence of canonical transformations involving scalings,
 a modulational Ansatz  as well as  homogenization
techniques that preserve the Hamiltonian character of the problem.
 A central tool in this approach is the Dirichlet--Neumann operator that appears naturally in the  Hamiltonian (total energy) of  the water wave system.
 It has a convergent Taylor series in terms of the surface elevation \cite{CS93} which in turn provides an expansion of the Hamiltonian
 for small-amplitude waves.  This analysis involves the construction of a Birkhoff  normal form transformation that
eliminates non-resonant cubic terms leading to a reduced  Hamiltonian at fourth order.  
The resulting Dysthe equation is Hamiltonian and has differences in the high-order nonlinear terms as compared to the original equation
 of \cite{D79}.  Furthermore in  \cite{CGS21}, this Hamiltonian Dysthe equation was tested against direct numerical simulations of the Euler system and very good agreement
 was obtained. For this purpose, one needs to reconstruct the surface elevation from the solution of the envelope equation. 
While classically this reconstruction is carried out perturbatively in terms of a Stokes expansion \cite{GT11},  the procedure in \cite{CGS21} is achieved through  a non-perturbative method that requires  solving a Burgers equation associated to the cubic Birkhoff normal form transformation.
  In subsequent work \cite{GKSX21}, an alternate spatial version of this Hamiltonian Dysthe equation, well adapted  for comparison with laboratory experiments,
  was derived and tested against experimental results on periodic groups and short-wave packets as previously discussed by Lo and Mei \cite{LM85}.
 
The purpose of this paper is to extend this analysis to the three-dimensional problem of gravity waves on deep water.
Our new contributions are two-fold. First, we present the derivation of a Hamiltonian Dysthe equation 
through a sequence of canonical transformations that preserve the Hamiltonian character of the system, 
starting from the three-dimensional Euler equations for an irrotational ideal fluid.  
The surface reconstruction also involves solving a Hamiltonian system of differential equations.
As a consequence, the entire solution process fits within a Hamiltonian framework.
Second, we test our model against direct numerical simulations of the three-dimensional Euler system
and against numerical solutions of the classical Dysthe equation.
We propose a simplified version of our approach for surface reconstruction that is more efficient numerically
by exploiting the disparity in length scales between the longitudinal and transverse wave dynamics.
 
 The  paper is organized as follows.   In Section 2,  the  mathematical formulation of three-dimensional deep-water water waves
as a Hamiltonian system for the surface elevation and trace of the velocity potential is recalled. Section 3 provides the basic tools of Birkhoff normal form transformations, and in particular the third-order normal form that eliminates all cubic terms 
from the Hamiltonian is obtained.  In Section 4, we calculate the new Hamiltonian truncated at fourth order and introduce the modulational
Ansatz  where approximate solutions take the form of  weakly modulated monochromatic waves, leading to a Hamiltonian
Dysthe equation for the wave envelope as described in Section 5. Section 6 is devoted to the reconstruction of the free surface, which includes
inverting the third-order normal form transformation. In Section 7, we perform a modulational stability analysis for Stokes wave solutions.
Finally, we present numerical tests in Section 8.

\section{The water wave system}

We consider a three-dimensional fluid  in a domain of infinite depth
$
S(\eta,t) = \{(\bx, z): \bx = (x, y) \in \R^2, -\infty < z < \eta(\bx, t)\},
$
where $z=\eta(\bx, t)$ represents the free surface at time $t$. 
Assuming the fluid is incompressible, inviscid and irrotational, it is described by a potential flow such that the velocity field ${\rm u}(\bx,z,t) =\nabla \varphi$ satisfies
\begin{equation*}\label{Eq:LaplacesEqn}
   \Delta \varphi =0 \,, 
\end{equation*}
in the fluid domain $S(\eta,t)$. 
On the surface $\{ z = \eta(\bx,t) \}$, 
two boundary conditions are imposed, namely
\begin{eqnarray*}
& \partial_t \eta = \partial_z \varphi - \partial_\bx \eta \cdot \partial_\bx \varphi \,, \quad
 \partial_t \varphi+ \frac{1}{2} |\nabla \varphi|^2 + g\eta = 0 \,,
\nonumber
\end{eqnarray*}
where $g$ is the acceleration due to gravity.
The symbol $\nabla$ denotes the spatial gradient $(\partial_\bx, \partial_z)$ when applied to functions
or the variational gradient when applied to functionals.


\subsection{Hamiltonian formulation}

It is known since the seminal paper of Zakharov \cite{Z68} that the water wave system has a canonical
Hamiltonian formulation with conjugate variables $(\eta(\bx,t), \xi(\bx,t) := \varphi(\bx, \eta(\bx,t),t))$ such that
\begin{equation}
\label{ww-hamiltonian-equation}
\partial_t \begin{pmatrix}
\eta \\ \xi
\end{pmatrix} =  J \, \nabla H(\eta,\xi)  = \begin{pmatrix}
0 & 1 \\ -1 & 0
\end{pmatrix} \begin{pmatrix}
\partial_\eta H \\ \partial_\xi H
\end{pmatrix} \,,
\end{equation}
where the Hamiltonian $H(\eta, \xi)$ is the total energy and is  expressed in terms of the Dirichlet--Neumann operator (DNO) $G(\eta)$ as
\begin{equation}
\label{hamiltonian}
H(\eta, \xi) = \frac{1}{2} \int_{\bbR^2} \left( \xi G(\eta) \xi + g \eta^2 \right) d\bx \,.
\end{equation}
This operator 
is defined as a map which associates to the Dirichlet data $\xi$ the normal derivative of the harmonic function $\varphi$ at the surface with
a normalizing factor, namely 
$$
G(\eta) : \xi \longmapsto \sqrt{1+|\partial_\bx \eta|^2} \, \partial_{\bn} \varphi \big|_{z=\eta} \,.
$$
It is analytic in $\eta$ \cite{CM85} and admits a convergent Taylor series expansion
\begin{equation} 
\label{series}
G(\eta) = \sum_{m=0}^\infty G^{(m)}(\eta) \,,
\end{equation}
 about $\eta=0$. For each $m$, $G^{(m)}(\eta)$ is homogeneous of degree $m$ in $\eta$ and can be calculated explicitly via  recursive relations \cite{CS93}.  Denoting  $D = -i \, \partial_{\bx}$, the first three terms are
\begin{equation}
\label{g-012-recursive}
\left\{ \begin{array}{l}
G^{(0)}(\eta) = |D| \,, \\
G^{(1)}(\eta) = D \cdot \eta D - G^{(0)}\eta G^{(0)} \,, \\
G^{(2)}(\eta) = -\frac{1}{2} \left( |D|^2 \eta^2 G^{(0)}+ G^{(0)} \eta^2 |D|^2 - 2G^{(0)} \eta G^{(0)} \eta G^{(0)} \right) \,.
\end{array} \right.
\end{equation}
We denote the Fourier transform of   the  real-valued pair 
$(\eta (\bx), \xi (\bx))$ by
\begin{equation*}
\label{fourier-eta-xi-def}
(\eta_\rk, \xi_\rk) = \frac{1}{2\pi} \int_{\RR^2} e^{-i\rk \cdot \bx} 
(\eta(\bx), \xi(\bx)) \, d\bx \,,
\end{equation*}
where, for simplicity, we have dropped the usual ``hat'' notation as well as the time dependence.
In Fourier variables, the water wave system also has the form of a  canonical Hamiltonian system (see Appendix \ref{App-0})
\begin{equation}
\label{ww-hamiltonian-equation-Fourier}
\partial_t \begin{pmatrix}
\eta_{-\rk} \\ \xi_{-\rk}
\end{pmatrix}
 = \begin{pmatrix}
0 & 1 \\ -1 & 0
\end{pmatrix} \begin{pmatrix}
\partial_{\eta_\rk} H \\ \partial_{\xi_\rk} H
\end{pmatrix} \,.
\end{equation}
Substituting the expansion for $G(\eta)$ into the Hamiltonian (\ref{hamiltonian}), we get
\begin{equation} \label{HH}
H = H^{(2)} + H^{(3)}+ H^{(4)} + \ldots
\end{equation}
 where each term $H^{(m)}$ is homogeneous of degree $m$ in the $(\eta, \xi)$ variables. Using  (\ref{g-012-recursive}), 
 the first three  terms  of $H$, written in Fourier variables, are
\begin{equation}
\label{h-fourier}
\begin{aligned}
H^{(2)}(\eta, \xi) & = \frac{1}{2} \int_{\RR ^2}
\left( |\rk| |\xi_\rk|^2 + g |\eta_\rk|^2 \right) d\rk \,, \\[2pt]
H^{(3)}(\eta, \xi) & = -\frac{1}{4\pi} \int_{\RR^6}
({\rm k}_1 \cdot {\rm k}_3 + |{\rm k}_1| |{\rm k}_3|) \xi_1 \eta_2 \xi_3
\delta_{123} d{\rm k}_{123} \,, \\[2pt]
H^{(4)}(\eta, \xi) & = -\frac{1}{16 \pi^2} \int_{\RR^8} 
|{\rm k}_1||{\rm k}_4| \left( |{\rm k}_1|+|{\rm k}_4|-2|{\rm k}_3+{\rm k}_4| \right)
\xi_1 \eta_2 \eta_3 \xi_4 
\delta_{1234} d{\rm k}_{1234} \,.
\end{aligned}
\end{equation}
In the above expressions, we have used the compact notations 
$\rk= (k_x, k_y)$, ${\rm k}_j = (k_{jx}, k_{jy}) \in \RR^2$,
$d{\rm k}_{123} = d{\rm k}_1 d{\rm k}_2 d{\rm k}_3$,  
$(\xi_j, \eta_j) = (\xi_{{\rm k}_j}, \eta_{{\rm k}_j})$
 and 
 $\delta_{1 \ldots n} = \delta({\rm k}_1+ \ldots + {\rm k}_n)$ where 
$
\delta(\rk) = \left( \frac{1}{2\pi} \right)^2
\int e^{- i\rk \cdot \bx} d\bx
$
is the Dirac distribution in two dimensions.
Hereafter, the domain of integration is omitted in integrals and is understood to be $\RR^2$ for each $\bx_j$ or ${\rm k}_j$.

\subsection{Complex symplectic coordinates and Poisson brackets}

The linear dispersion relation for deep-water gravity waves is
$\omega_\rk^2 = g|\rk|$. 
It is convenient to introduce the complex symplectic coordinates
\begin{equation}
\label{zeta-variable-fourier}
\begin{pmatrix}
z_\rk \\ \bar z_{-\rk}
\end{pmatrix} = P_1 \begin{pmatrix}
\eta_\rk \\ \xi_\rk
\end{pmatrix} = \frac{1}{\sqrt 2} \begin{pmatrix}
a_\rk & i \, a_\rk^{-1} \\
a_\rk & -i \, a_\rk^{-1} 
\end{pmatrix} \begin{pmatrix}
\eta_\rk \\ \xi_\rk
\end{pmatrix} \,,
\end{equation}
where $a_\rk^2 := \sqrt{g/|\rk|}= g/\omega_\rk$ and considering that the functions $\eta(\bx)$ and $\xi(\bx)$ are real-valued.
In these variables, the system \eqref{ww-hamiltonian-equation} reads
\begin{equation}
\label{z-physical-ham-system}
\partial_t \begin{pmatrix}
z_\rk \\ \bar z_{-\rk}
\end{pmatrix} = J_1 \begin{pmatrix}
\partial_{z_\rk} H \\ \partial_{\bar z_{-\rk}} H
\end{pmatrix} = \begin{pmatrix}
0 & -i \\ i & 0
\end{pmatrix} \begin{pmatrix}
\partial_{z_\rk} H \\ \partial_{\bar z_{-\rk}} H
\end{pmatrix} \,,
\end{equation}
with  $J_1 = P_1 J P_1^*$ \cite{CGS10}, where the star denotes the adjoint with respect to the $L^2$-scalar product.
The quadratic term $H^{(2)}$ becomes
\begin{equation*}
\label{quadratic-term-fourier}
H^{(2)} = \int \omega_\rk |\zeta_\rk|^2 d\rk \,,
\end{equation*}
while  the cubic term $H^{(3)}$  takes the form
\begin{equation}\label{h3}
H^{(3)} = \frac{1}{8\pi  \sqrt{2}} \int
      ( {\rm k}_1 \cdot {\rm k}_3 +|{\rm k}_1||{\rm k}_3|) \frac{a_1 a_3}{a_2} 
      (z_1-\bar z_{-1})(z_2+\bar z_{-2})(z_3-\bar z_{-3})  \delta_{123} d{\rm k}_{123} \,,
\end{equation}
where $z_{\pm j} := z_{\pm {\rm k}_j}$ and $a_j := a_{{\rm k}_j}$. 

The Poisson bracket of two functionals $K(\eta, \xi)$ and $H(\eta, \xi)$ 
of real-valued functions $\eta$ and $\xi$  is defined as
\begin{equation*}
\label{poisson-bracket}
\{K, H\} = \int (\partial_\eta H \partial_\xi K - \partial_\xi H \partial_\eta K) \, d\bx \,.
\end{equation*}
Assuming that $K$ and $H$ are real-valued, we have 
$$
\begin{aligned}
&\{K, H\}  = \int (\partial_{\eta_\rk} H \overline{\partial_{\xi_\rk} K} - \partial_{\xi_\rk} H \overline{\partial_{\eta_\rk} K}) d\rk = \int (\partial_{\eta_\rk} H \partial_{\bar \xi_\rk} K - \partial_{\xi_\rk} H \partial_{\bar \eta_\rk} K) d\rk \\ 
& = \int (\partial_{\eta_\rk} H \delta_{\xi_{-\rk}} K - \partial_{\xi_\rk} H \partial_{\eta_{-\rk}} K) d\rk   
= \int (\partial_{\eta_{{\rm k}_1}} H \partial_{\xi_{{\rm k}_2}} K - \partial_{\xi_{{\rm k}_1}} H \partial_{\eta_{{\rm k}_2}} K) \delta_{12} d{\rm k}_{12} \,.
\end{aligned}
$$
In terms of the complex symplectic coordinates, the Poisson bracket  is
\begin{equation}
\label{poisson-bracket-z}
\{K, H\} = \frac{1}{i} \int (\partial_{z_{{\rm k}_1}} H \partial_{\bar z_{-{\rm k}_2}} K - \partial_{\bar z_{- {\rm k}_1}} H \partial_{z_{{\rm k}_2}} K) \delta_{12} d{\rm k}_{12} \,.
\end{equation}

\section{Transformation theory} 

A monomial  $z_1 \ldots z_l \bar z_{-(l+1)} \ldots \bar z_{-m}$  in the Hamiltonian   
is resonant of order $m$ if
$$
\sum_{j=1}^l \omega_{{\rm k}_j} - \sum_{j=l+1}^m \omega_{{\rm k}_j} = 0 \quad \text{ for some }
{\rm k}_1+{\rm k}_2+\ldots+{\rm k}_m=0 \,.
$$
It is known that, for pure gravity waves on deep water, there are no resonant triads, that is  no triplets 
$({\rm k}_1,{\rm k}_2,{\rm k}_3)  \in {(\R^2)}^3$ with ${\rm k}_j \ne 0$ such that ${\rm k}_1+{\rm k}_2+{\rm k}_3=0$ and 
$\omega_{{\rm k}_1} \pm \omega_{{\rm k}_2} \pm \omega_{{\rm k}_3} = 0$.
This is because  $\omega(\rk)$  is an increasing concave function of $|\rk|$.

\subsection{Canonical transformation}

To eliminate non-resonant terms as they are not crucial at describing the wave dynamics, we look for a canonical transformation of physical variables
$$
\tau: w = \begin{pmatrix}
\eta \\ \xi
\end{pmatrix} \longmapsto w' \,,
$$
defined in a neighborhood of the origin, such that the transformed Hamiltonian satisfies
$$
H'(w') = H(\tau^{-1}(w')) \,, \quad 
\partial_t w' = J \, \nabla H'(w') \,,
$$
and reduces to
$$
H'(w') = H^{(2)}(w') +  Z^{(3)} + Z^{(4)} + \ldots +  Z^{(m )} + R^{(m+1)} \,,
$$
where 
$Z^{(m )}$ consists only of resonant terms  and $R^{(m+1)}$ is the remainder term \cite{CGS21b}. 
We construct the transformation $\tau$  by the Lie transform method as a Hamiltonian flow $\psi$ from ``time'' $s=-1 $ to ``time'' $s=0$ governed by
$$
\partial_s \psi = J \, \nabla K(\psi) \,, \quad \psi(w')|_{s=0} = w' \,, \quad 
\psi(w')|_{s=-1} = w \,,
$$
and associated to an auxiliary Hamiltonian $K$.
Such a transformation is canonical and preserves the Hamiltonian structure of the system. The Hamiltonian $H'$ satisfies 
$H'(w') = H(\psi(w'))|_{s=-1}$ and its Taylor expansion around $s=0$ is
$$
H'(w') = H(\psi(w'))|_{s=0} - \frac{d H}{ds}(\psi(w'))|_{s=0} + \frac{1}{2} \frac{d^2 H}{ds^2}(\psi(w'))|_{s=0} - \ldots
$$
Abusing  notations, we further use $w = (\eta,\xi)^\top $ 
to denote the new variable $w'$.  Terms in this expansion can be expressed using Poisson brackets as 
\begin{equation*}
\begin{aligned}
H(\psi(w))|_{s=0} & = H(w) \,, \\
\frac{d H}{ds}(\psi(w))|_{s=0} & = \int \left( \partial_\eta H \partial_s \eta + 
\partial_\xi H \partial_s \xi \right) d\bx \\    
& = \int \left( \partial_\eta H \partial_\xi K - \partial_\xi H \partial_\eta K \right) d\bx = 
\{K, H\}(w) \,,
\end{aligned}
\end{equation*}
and similarly for the remaining terms. The Taylor expansion of $H'$ around $s=0$ now has the form 
$$
H'(w) = H(w) - \{K, H\}(w) + \frac{1}{2} \{K, \{K, H\}\}(w) - \ldots
$$
Substituting this transformation into  the expansion (\ref{HH}) of $H$, we obtain
\begin{align*}
   H'(w) & = H^{(2)}(w) + H^{(3)}(w) + \ldots \\
 & \quad - \{K,H^{(2)} \} (w) -\{K, H^{(3)}\} (w)  -\{K, H^{(4)}\} (w) - \ldots \\
 & \quad + \frac{1}{2}\{K,\{K,H^{(2)} \}\} (w) +\frac{1}{2} \{K,\{K,H^{(3)}\}\} (w) + \ldots  
\end{align*} 
If $K$ is homogeneous of degree $m$ and $H^{(n)}$ is homogeneous of degree $n$, then $\{K,H^{(n)} \}$ is of degree $m+n-2$. 
Thus, if we construct an auxiliary Hamiltonian $K=K^{(3)}$ that is homogeneous of degree $3$ and satisfies the relation 
\begin{equation}\label{cohomological-relation}
   H^{(3)}- \{K^{(3)},H^{(2)} \}  = 0 \,,
\end{equation}
we will have eliminated all cubic terms from the transformed Hamiltonian $H'$.
We can repeat this process at each order.

\subsection{Third-order Birkhoff normal form}

To find the auxiliary Hamiltonian $K^{(3)}$ from \eqref{cohomological-relation}, 
we use the following diagonal property of the coadjoint operator ${\rm coad}_{H^{(2)}} := \{\cdot, H^{(2)}\}$ when applied to monomial terms \cite{CS16}.
For example,
taking 
$
\mathcal{I} := \int z_1 z_2 \bar z_{-3} \delta_{123} d{\rm k}_{123},
$
we have 
\begin{equation} \label{ad-H2}
\{ \mathcal{I}, H^{(2)}\}  =  
i \int (\omega_1 + \omega_2 - \omega_3)   z_1 z_2 \bar z_{-3}  \delta_{123} d{\rm k}_{123} \,,
\end{equation}
where $\omega_j := \omega_{\rk_j}$. We will employ such notations throughout the paper when no confusion should arise. 

\begin{proposition} 
\cite{CS16} The cohomological equation \eqref{cohomological-relation} has a unique solution $K^{(3)}$ which, in complex symplectic coordinates, is 
\begin{align} \label{K3-fourier-z}
   K^{(3)} = &~ \frac{1}{8i{\pi}\sqrt{2}} \int
      ({\rm k}_1 \cdot {\rm k}_3 + |{\rm k}_1| |{\rm k}_3|)  \Big[  \frac{a_1a_3}{a_2} 
     \frac{z_1 z_2 z_3- \bar z_{-1} \bar z_{-2} \bar z_{-3} }
     {\omega_1 + \omega_2+ \omega_3} \\
     \nonumber
  \\ \nonumber
   & -   
    2 \frac{a_1a_3}{a_2} 
       \frac{z_1 z_2 \bar z_{-3}- \bar z_{-1} \bar z_{-2} z_3}{\omega_1+ \omega_2 -\omega_3}   
    +   
    \frac{a_1a_3}{a_2} 
      \frac{ z_1\bar z_{-2}  z_3-  \bar z_{-1} z_2 \bar z_{-3} }{\omega_1 -\omega_2 +\omega_3}  \Big] \delta_{123} d{\rm k}_{123} \,.
   \nonumber
\end{align}
Alternatively, in the $(\eta,\xi)$ variables, $K^{(3)}$  has the form
\begin{eqnarray}
\label{K3-fourier}
K^{(3)} &=& \frac{1}{2\pi} \int 
\frac{{\rm k}_1 \cdot {\rm k}_3 + |{\rm k}_1| |{\rm k}_3|}{d_{123}} 
\Big[ g^2 (|{\rm k}_1|-|{\rm k}_2|-|{\rm k}_3|) \eta_1 \eta_2 \xi_3 + 
g^2 |{\rm k}_2| \eta_1 \xi_2 \eta_3  \\
& \phantom{t} & 
\qquad + \frac{g}{2} |{\rm k}_2|
(|{\rm k}_1| - |{\rm k}_2| + |{\rm k}_3|)  
\xi_1 \xi_2 \xi_3 \Big] \delta_{123} d{\rm k}_{123} \,,
\nonumber
\end{eqnarray}
where the  denominator $d_{123}$ is given by 
\begin{align*}\label{d-123}
d_{123} &= d(\omega_1,\omega_2,\omega_3)= ( \omega_1 +\omega_2+\omega_3)
  (\omega_1 +\omega_2-\omega_3)( \omega_1 -\omega_2+\omega_3)( \omega_1 -\omega_2-\omega_3) \nonumber \\
& =g^2 \left( 
|{\rm k}_1|^2 + |{\rm k}_2|^2 + |{\rm k}_3|^2 - 
2|{\rm k}_1||{\rm k}_2| - 2|{\rm k}_2||{\rm k}_3| - 
2|{\rm k}_1||{\rm k}_3| \right) \nonumber \\
& =- 2 g^2 \left( {\rm k}_1 \cdot {\rm k}_2 + 
{\rm k}_2 \cdot {\rm k}_3 + 
{\rm k}_3 \cdot {\rm k}_1+ |{\rm k}_1| |{\rm k}_2| + |{\rm k}_2| |{\rm k}_3|  + |{\rm k}_3| |{\rm k}_1|\right) \,.
\end{align*}
\end{proposition}

\begin{proof}
We write $H^{(3)}$ given in (\ref{h3})  as a linear combination of third-order monomials in $z_\rk$ and $\bar{z}_{-\rk}$.
Identity \eqref{ad-H2} allows us to solve the cohomological equation \eqref{cohomological-relation}. We identify terms to obtain \eqref{K3-fourier-z}. 
The expression \eqref{K3-fourier} is derived from substitution of the relations \eqref{zeta-variable-fourier}.
 It will be useful later for the reconstruction of the surface elevation.
Notice that  the denominator $d_{123}$ is symmetric with respect to any permutation of the indices.
\end{proof}

\begin{remark}
The absence of resonant triads implies that the denominators do not vanish if none of the wavenumbers ${\rm k}_1$, ${\rm k}_2$, ${\rm k}_3$ vanish. 
However, in these integrals,
any of the ${\rm k}_j$ could be arbitrarily close to 0. In such a case, the denominator becomes small   but is compensated by a small numerator. For example, if $|{\rm k}_1| \sim |{\rm k}_2|$ are $\calO(1)$ while $|{\rm k}_3|$ is small, the denominator 
$d_{123} \sim  |{\rm k}_3|$ while, in the numerator, the term $|{\rm k}_1 \cdot {\rm k}_3 + |{\rm k}_1| |{\rm k}_3|| \sim |{\rm k}_3|$. 
On the other hand, if $|{\rm k}_1| \sim |{\rm k}_3|$ are $\calO(1)$ while $|{\rm k}_2|$ is small, the denominator 
$d_{123} \sim  |{\rm k}_2|$ and we can use that $| |{\rm k}_1|- |{\rm k}_3| | \le |{\rm k}_1+{\rm k}_3|= |{\rm k}_2|$ to compensate it. In all cases, all integrals are convergent.
\end{remark}


The third-order normal form defining the new coordinates is obtained as the solution map at $s=0$ of 
the Hamiltonian flow 
\begin{equation*}
\label{NFT-flow-fourier-initial}
\partial_s   
\begin{pmatrix}
\eta \\ \xi
\end{pmatrix} = 
\begin{pmatrix}
0 & 1 \\
-1 & 0
\end{pmatrix}
\begin{pmatrix}
\partial_{\eta}K^{(3)} \\
\partial_{\xi} K^{(3)}
\end{pmatrix} \,,
\end{equation*}
with initial condition at $s= -1$ being the original variables $(\eta,\xi)$.  Equivalently, in Fourier coordinates,
\begin{align}             \label{Eq:Hamiltonflow-Fourier}\               
   \partial_s \eta_{-\rk} = \partial_{\xi_\rk}
   K^{(3)} \,,  \quad 
   \partial_s \xi_{ -\rk} 
     =  - \partial_{\eta_\rk} K^{(3)} \,,
     \end{align}
where
\begin{align*}
 \partial_{\xi_\rk}   K^{(3)} & = \frac{1}{2\pi} \int   P_{12\rk}    \eta_1 \eta_2 \delta_{1\rk2} d{\rm k}_{12}
+ Q_{\rk23} \xi_2 \xi_3 \delta_{2\rk3} d{\rm k}_{23} \,, \\ 
 \partial_{\eta_\rk}   K^{(3)} & = \frac{1}{2\pi} \int   R_{\rk23}    \eta_2 \xi_3 \delta_{\rk23} d{\rm k}_{23} \,, 
\end{align*}
by virtue of \eqref{K3-fourier}. The coefficients in the above integrals are given by
\begin{eqnarray*}
P_{12\rk} & = & -\frac{1}{2} |\rk| + \frac{g^2}{d_{12\rk}}
({\rm k}_1 \cdot \rk + |{\rm k}_1| |\rk|)  (|{\rm k}_1|-|{\rm k}_2|-  3|\rk|) \,, \\ 
Q_{\rk23} & = & -\frac{1}{4g} \big[ 2( \rk \cdot {\rm k}_3 + |\rk| |{\rm k}_3|)  +( {\rm k}_2 \cdot {\rm k}_3 + |{\rm k}_2| |{\rm k}_3|) \big] \cr
& & -\frac{g}{2d_{\rk23}}   \big[ 2 (\rk \cdot {\rm k}_3 + |\rk| |{\rm k}_3|)^2
  + ({\rm k}_2 \cdot {\rm k}_3 + |{\rm k}_2| |{\rm k}_3|)^2 \big] \,, \\ 
 R_{\rk23} & = & -|{\rm k}_3| + \frac{g^2}{d_{\rk23}}  \Big[ (\rk \cdot {\rm k}_3 + |\rk| |{\rm k}_3|) (|\rk|-|{\rm k}_2|-  3|{\rm k}_3|)  \cr
& & + ({\rm k}_2 \cdot {\rm k}_3 + |{\rm k}_2| |{\rm k}_3|)  (|{\rm k}_2|-|\rk|-  3|{\rm k}_3|)  \Big] \,. 
\end{eqnarray*}

\section{Reduced Hamiltonian}

After applying the third-order normal form transformation, the new Hamiltonian $H'$ becomes (with the prime dropped)
\begin{equation*}
\label{new-hamiltonian}
\begin{aligned}
H(w) & = H^{(2)}(w) + H^{(4)}(w) - \{K^{(3)}, H^{(3)}\}(w) + 
\frac{1}{2} \{K^{(3)}, \{K^{(3)}, H^{(2)}\}\}(w) + R^{(5)}\\
& = H^{(2)}(w) + H_+^{(4)}(w) + R^{(5)} \,,
\end{aligned}
\end{equation*}
where $R^{(5)}$ denotes all terms of order $5$ and higher, and $H^{(4)}_+$ is the new fourth-order term
\begin{equation}
\label{new-H4-formula}
H^{(4)}_+ = H^{(4)} - \frac{1}{2} \{K^{(3)}, H^{(3)}\} \,.
\end{equation}

\subsection{Fourth-order term $H_+^{(4)}$}

Let
\begin{equation}
\label{A-123}
S_{123} := ({\rm k}_1 \cdot {\rm k}_3 + |{\rm k}_1| |{\rm k}_3|) \frac{a_1 a_3}{a_2} \,, \quad  
A_{123} := \frac{1}{8 \pi \sqrt{2}} \left( S_{123} + S_{312} - S_{231} \right) \,.
\end{equation}
\begin{lemma}\label{Hamiltonians-in-z}
We have
\begin{equation}
\label{H3-in-z-2}
H^{(3)} = \int_{\reals^6}  A_{123} (z_1 z_2 z_3 + \bar z_{-1} \bar z_{-2} \bar z_{-3} -
z_1 z_2 \bar z_{-3} - \bar z_{-1} \bar z_{-2} z_3)  \delta_{123} d{\rm k}_{123} \,,
\end{equation}
\begin{equation}
\label{K3-in-z-2}
K^{(3)} = 
 \frac{1}{i} \int_{\reals^6}  A_{123} \Big[ \frac{z_1 z_2 z_3 - \bar z_{-1} \bar z_{-2} \bar z_{-3}}{\omega_1 + \omega_2 + \omega_3} 
- \frac{z_1 z_2 \bar z_{-3} - \bar z_{-1} \bar z_{-2} z_3}{\omega_1 + \omega_2 - \omega_3}   \Big] \delta_{123} d{\rm k}_{123} \,.
\end{equation}
\end{lemma}

\begin{proof} 
Expanding the brackets in \eqref{h3}, we find 
\begin{align} \label{h3-expanded}
H^{(3)} & = \frac{1}{8\pi \sqrt{2}} \int S_{123} \Big[( z_1 z_2 z_3 + \bar z_{-1} \bar z_{-2} \bar z_{-3}) 
 -   ( z_1 z_2 \bar z_{-3} + \bar z_{-1} \bar z_{-2} z_3 ) \nonumber\\
& \quad + (z_1 \bar z_{-2} z_3 + \bar z_{-1} z_2 \bar z_{-3})
 -  ( z_1 \bar z_{-2} \bar z_{-3} + \bar z_{-1} z_2 z_3 ) \Big]\delta_{123} d{\rm k}_{123} \,.
\end{align}
By symmetry, the first term on the RHS of (\ref{h3-expanded}) identifies to the two  terms of  (\ref{H3-in-z-2}). 
Applying  the index rearrangement $(1,2,3)\to (2,3,1)$ to the third term of (\ref{h3-expanded}) and $(1,2,3)\to(3,1,2)$ to its fourth term, 
the last three terms on the RHS of (\ref{h3-expanded}) reduce to the last two terms of (\ref{H3-in-z-2}). 
\end{proof}
The modified fourth-order term $H_+^{(4)}$  given in (\ref{new-H4-formula}) is  the sum of integrals with all possible  combinations of 
fourth-order monomials in   $z_\rk$ and $\bar z_{-\rk}$, that is 
\begin{equation}
\label{new-H4-general-form}
\begin{aligned}
H_+^{(4)} = &
\int  \Big [T^{+} z_1 z_2 z_3 z_4  + 
 T^\pm z_1 z_2 z_3 \bar z_{-4}   + 
 T^{+}_{-} z_1 z_2 \bar z_{-3} \bar z_{-4}  + 
T^\mp z_1 \bar z_{-2} \bar z_{-3} \bar z_{-4} \\
 & \quad + T^{-} \bar z_{-1} \bar z_{-2} \bar z_{-3} \bar z_{-4} \Big] \delta_{1234} d{\rm k}_{1234} \,.
\end{aligned}
\end{equation}
In view of the forthcoming modulational Ansatz and homogenization process,  it is however not necessary to calculate
explicitly all the coefficients above. We only need the coefficient $T^+_-$ of monomial $z_1 z_2 \bar z_{-3} \bar z_{-4}$.
We thus denote by
\begin{equation*} \label{decompH4}
 H_R^{(4)} = \int T_1 z_1 z_2 \bar z_{-3} \bar z_{-4} \delta_{1234} d{\rm k}_{1234} \,, \quad
  \{K^{(3)}, H^{(3)}\}_R = \int T_2 z_1 z_2 \bar z_{-3} \bar z_{-4} \delta_{1234} d{\rm k}_{1234} \,,
\end{equation*}
 the contributions from these monomials to $H^{(4)}$  and $\{K^{(3)}, H^{(3)}\}$ respectively, and
\begin{equation}
\label{fourth-order-H-R}
 H_{+R}^{(4)} =  \int T^{+}_{-} z_1 z_2 \bar z_{-3} \bar z_{-4} \delta_{1234} d{\rm k}_{1234} \,,
 \end{equation}
 with  $H_{+R}^{(4)} =  H_R^{(4)} -\frac{1}{2} \{K^{(3)}, H^{(3)}\}_R$ and
$T^+_- = T_1-\frac{1}{2} T_2$.

\begin{proposition} \label{T1}
The coefficient $T_1$ has the form
\begin{eqnarray*}
\label{I-term0}
&T_1 = 
      \frac{\sqrt[4]{|{\rm k}_1||{\rm k}_2||{\rm k}_3||{\rm k}_4|}}{64 \pi^2}   \Big[\sqrt{|{\rm k}_1||{\rm k}_2|} \big(|{\rm k}_1|+|{\rm k}_2|-2|{\rm k}_2-{\rm k}_3|\big) 
  + \sqrt{|{\rm k}_3||{\rm k}_4|}\big( |{\rm k}_3|+|{\rm k}_4|-2|{\rm k}_1-{\rm k}_4| \big) \\
&-2\sqrt{|{\rm k}_1||{\rm k}_4|} \big( 2|{\rm k}_1|+2|{\rm k}_4|-|{\rm k}_1+{\rm k}_2|-|{\rm k}_3+{\rm k}_4|- |{\rm k}_1-{\rm k}_3| - |{\rm k}_2-{\rm k}_4| \big) \Big] \,.
\end{eqnarray*}
\end{proposition}

\begin{proof}
Write  $H^{(4)}$ given in  (\ref{h-fourier}) in terms of complex symplectic coordinates
$$
H^{(4)} = \int D_{1234} (z_1 - \bar z_{-1}) (z_2 + \bar z_{-2}) (z_3 + \bar z_{-3}) (z_4 - \bar z_{-4}) \delta_{1234} d{\rm k}_{1234} \,,
$$
where 
\[
D_{1234} = \frac{\sqrt[4]{|{\rm k}_1||{\rm k}_2||{\rm k}_3||{\rm k}_4|}}{64 \pi^2} \sqrt{|{\rm k}_1||{\rm k}_4|} (|{\rm k}_1| + |{\rm k}_4| - 2|{\rm k}_3+{\rm k}_4|) \,.
\]
 Extracting terms of type ``$zz\bar z\bar z$", 
we have
$$
\begin{aligned}
H^{(4)}_{R} & = \int D_{1234} ( -z_1 z_2 \bar z_{-3} \bar z_{-4} - \bar z_{-1} \bar z_{-2} z_3 z_4
 - z_1 \bar z_{-2}  z_{3} \bar z_{-4} - \bar z_{-1} z_2 \bar z_{-3} z_{4}\\
& \qquad + z_1 \bar z_{-2} \bar z_{-3} z_{4} + \bar z_{-1} z_2 z_{3} \bar z_{-4}) \delta_{1234} d{\rm k}_{1234} \,.
 \end{aligned}
$$
This integral  can alternatively be written after index rearrangements as
{\small{
\[
H^{(4)}_{R} = \int (-D_{1234} - D_{4321} - D_{1324} - D_{4231} + D_{1432} + D_{3214}) z_1 z_2 \bar z_{-3} \bar z_{-4} \delta_{1234} d{\rm k}_{1234} \,, %
\]
}}
with the  coefficient above equal to  $T_1$ as given in Proposition \ref{T1}.
\end{proof}

\begin{proposition}
\label{lemma-T2-coeff}
 Denoting
 $\ell_{{\rm k}_1}^{ {\rm k}_2} := \frac{|{\rm k}_1||{\rm k}_2|+{\rm k}_1 \cdot {\rm k}_2}{\sqrt{|{\rm k}_1||{\rm k}_2|}}$,
we have 
$ T_2 = {\rm I} + {\rm II} + {\rm III}$ with
\begin{eqnarray}
\label{I-term-new}
{\rm I} & = & \frac{g^{1/2}}{128 \pi^2}  \big(|{\rm k}_1||{\rm k}_2||{\rm k}_3||{\rm k}_4| |{\rm k}_1+{\rm k}_2| |{\rm k}_3+{\rm k}_4|\big)^{1/4}  
\big(\ell_{{\rm k}_1}^{{\rm k}_2} + \ell_{{\rm k}_1+{\rm k}_2}^{-{\rm k}_1} + \ell_{{\rm k}_1+{\rm k}_2}^{-{\rm k}_2}\big) \nonumber \\
& \phantom{t} &
 \times \big(\ell_{{\rm k}_3}^{{\rm k}_4} + \ell_{{\rm k}_3+{\rm k}_4}^{-{\rm k}_3} + \ell_{{\rm k}_3+{\rm k}_4}^{-{\rm k}_4}\big)  
\Big( \frac{1}{\omega_{{\rm k}_1}+ \omega_{{\rm k}_2}+ \omega_{{\rm k}_1+{\rm k}_2}} +\frac{1}{\omega_{{\rm k}_3}+ \omega_{{\rm k}_4}+ \omega_{{\rm k}_3+{\rm k}_4}} \Big) \,,
\end{eqnarray}
\begin{eqnarray}
\label{II-term-new}
{\rm II} & = & \frac{g^{1/2}}{32 \pi^2} (|{\rm k}_1||{\rm k}_2||{\rm k}_3||{\rm k}_4| |{\rm k}_1+{\rm k}_3| |{\rm k}_2+{\rm k}_4|)^{1/4}   
(\ell_{{\rm k}_1}^{ {\rm k}_3} + \ell_{{\rm k}_1+{\rm k}_3}^{-{\rm k}_3} - \ell_{{\rm k}_1+{\rm k}_3}^{-{\rm k}_1})
\nonumber \\
& \phantom{t} &
\times (\ell_{{\rm k}_4}^{{\rm k}_2}   + \ell_{{\rm k}_2+{\rm k}_4}^{-{\rm k}_2} - \ell_{{\rm k}_2+{\rm k}_4}^{-{\rm k}_4})  
 \Big( \frac{1}{\omega_{{\rm k}_1}+ \omega_{{\rm k}_1+{\rm k}_3}- \omega_{{\rm k}_3}} +\frac{1}{\omega_{{\rm k}_4}+ \omega_{{\rm k}_4+{\rm k}_2} - \omega_{{\rm k}_2}} \Big) \,,
\end{eqnarray}
\begin{eqnarray}
\label{III-term}
{\rm III} & = & - \frac{g^{1/2}}{128 \pi^2}   \big(|{\rm k}_1||{\rm k}_2||{\rm k}_3||{\rm k}_4| |{\rm k}_1+{\rm k}_2| |{\rm k}_3+{\rm k}_4|)^{1/4} 
\big(\ell_{{\rm k}_1+{\rm k}_2}^{-{\rm k}_1} + \ell_{{\rm k}_1+{\rm k}_2}^{-{\rm k}_2} - \ell_{{\rm k}_1}^{{\rm k}_2}\big) \nonumber \\
& \phantom{t} &
\times \big(\ell_{{\rm k}_3+{\rm k}_4}^{-{\rm k}_3} + \ell_{{\rm k}_3+{\rm k}_4}^{-{\rm k}_4}- \ell_{{\rm k}_3}^{{\rm k}_4}\big) 
\Big( \frac{1}{\omega_{{\rm k}_1}+ \omega_{{\rm k}_2}- \omega_{{\rm k}_1+{\rm k}_2}} +\frac{1}{\omega_{{\rm k}_3}+ \omega_{{\rm k}_4}- \omega_{{\rm k}_3+{\rm k}_4}} \Big) \,.
\end{eqnarray}

\end{proposition}
The proof is a little more technical so we present it in Appendix \ref{appendix-T2-coeff}.

\subsection{Modulational Ansatz and homogenization}  

We are interested in solutions in the form of near-monochromatic waves with carrier wavenumber $\rk_0 = (k_0, 0)$, $k_0>0$. 
In Fourier space,  this corresponds to a narrow band approximation where $\eta_k$ and $\xi_k$
are  localized near $\rk_0$. 
Accordingly, when dealing with variables of type $z_{{\rm k}_j}$ and $\overline z_{-\rk_j}$ as shown in (\ref{new-H4-general-form}), we introduce the modulational Ansatz 
\begin{equation}
\label{kj-around-k0}
\pm {\rm k}_j = \rk_0 + \epsilon \, \chi_j \,, \quad 
{\rm where} \quad
\chi_j = (\lambda_j, \mu_j) \,, \quad
\frac{|\chi_j|}{k_0} = \calO (1) \,, \quad 
\epsilon \ll 1 \,,
\end{equation}
respectively.
A scale-separation lemma will show that, in this regime, all integrals in (\ref{new-H4-general-form}), except the third one, are arbitrarily small as
$\epsilon \to 0$. The third integral will be used later to derive a suitable approximation for the fourth-order term $H_+^{(4)}$.

We introduce the function $U$ defined in the Fourier space as 
\begin{equation}
\label{U-in-fourier}
U(\chi) = \epsilon \, z(\rk_0 + \epsilon \chi) \,, \quad \overline U(\chi) = \epsilon \, \overline z(\rk_0 + \epsilon \chi) \,,
\end{equation}
and we employ the notation $U_j = U_{\chi_j} := U(\chi_j)$ when no confusion should arise (again the time dependence is omitted).
The first integral in (\ref{new-H4-general-form}) has the form
\begin{eqnarray}
\nonumber
\mathcal{I}_1 & := & \int T_{{\rm k}_1, {\rm k}_2, {\rm k}_3, {\rm k}_4} z_1 z_2 z_3 z_4 \delta_{1234} d{\rm k}_{1234}\\
\label{homogenization-integral}
& = & \frac{1}{4\pi^2} \iint T_{{\rm k}_1, {\rm k}_2, {\rm k}_3, {\rm k}_4} z_1 z_2 z_3 z_4 e^{-i ({\rm k}_1+{\rm k}_2+{\rm k}_3+{\rm k}_4) \cdot \bx} d{\rm k}_{1234} d\bx \,,
\nonumber
\end{eqnarray}
where we have used  the identity
$
\delta_{1234} = \frac{1}{4\pi^2} \int
e^{- i ({\rm k}_1+{\rm k}_2+{\rm k}_3+{\rm k}_4) \cdot \bx} d\bx.
$
After the change of variables (\ref{kj-around-k0}),  
 $\mathcal{I}_1$ becomes
\begin{eqnarray}
\nonumber
\mathcal{I}_1 
\label{homogenization-integral-new}
& = & \frac{\epsilon^4}{4\pi^2}   \int e^{- 4 i k_0 x}    \int  \widetilde T_{\chi_1,\chi_2,\chi_3,\chi_4}     U_1 U_2 U_3 U_4  
e^{-i (\chi_1+\chi_2+\chi_3+\chi_4) \cdot (\epsilon \bx)} d\chi_{1234} d\bx \,,
\end{eqnarray}
where $\widetilde T_{\chi_1,\ldots,\chi_4} =T_{k_0+\epsilon \chi_1, \ldots, k_0+\epsilon \chi_4}$.
The inner integral above
$$ 
\int  \widetilde T_{\chi_1,\chi_2,\chi_3,\chi_4}     U_1 U_2 U_3 U_4  
e^{-i (\chi_1+\chi_2+\chi_3+\chi_4) \cdot (\epsilon \bx)} d\chi_{1234} \,, 
$$
identifies to a function $f(\epsilon \bx)$. Thus, 
$\mathcal{I}_1 =\ \frac{\epsilon^4}{4\pi^2}   \int e^{-4 i k_0 x}  f(\epsilon \bx) d\bx.$

The second integral in (\ref{new-H4-general-form}) has the form 
\begin{equation*}
\begin{aligned}
\mathcal{I}_2 &  
=\int T_{{\rm k}_1, {\rm k}_2, {\rm k}_3, {\rm k}_4} z_1 z_2 z_3 \bar z_{-4} \delta_{1234} d{\rm k}_{1234} \\
&=\frac{1}{4 \pi^2} \iint T_{{\rm k}_1, {\rm k}_2, {\rm k}_3, {\rm k}_4} z_1 z_2 z_3 \bar z_{-4} e^{-i ({\rm k}_1+{\rm k}_2+{\rm k}_3+{\rm k}_4) \cdot \bx} d{\rm k}_{1234} d\bx \,.
\end{aligned}
\end{equation*}
After the change of variables (\ref{kj-around-k0}),
 $\mathcal{I}_2$ becomes
\begin{eqnarray}
\nonumber
\quad\mathcal{I}_2
& = & \frac{\epsilon^4}{4\pi^2}   \int e^{-2 i k_0 x}    \int  \widetilde T_{\chi_1,\chi_2,\chi_3,\chi_4}     U_1 U_2 U_3 \overline U_4  
e^{- i (\chi_1+\chi_2+\chi_3-\chi_4) \cdot (\epsilon \bx)} d\chi_{1234} d\bx \,,
\end{eqnarray}
where $\widetilde T_{\chi_1,\ldots,\chi_4} =T_{k_0+\epsilon \chi_1, \ldots, -k_0-\epsilon \chi_4}$.
Again, the inner integral 
$$ 
\int  \widetilde T_{\chi_1,\chi_2,\chi_3,\chi_4}     U_1 U_2 U_3 \overline U_4  
e^{- i (\chi_1+\chi_2+\chi_3-\chi_4) \cdot (\epsilon \bx)} d\chi_{1234} \,, 
$$
identifies to a function $f(\epsilon \bx)$.  Consequently, 
$\mathcal{I}_2 =\ \frac{\epsilon^4}{4\pi^2}   \int e^{-2 i k_0 x}  f(\epsilon \bx) d\bx.$
To evaluate  the  integrals $\mathcal{I}_1$ and $\mathcal{I}_2$, 
we use the following  scale-separation lemma.
\begin{lemma}
\label{scale-separation-lemma}
Let $f$ be a real-valued function of Schwartz class and $\alpha \in \RR^2$ be a nonzero constant vector. Then, for all $N$, 
\[
\int e^{i \alpha \cdot \bx} f(\epsilon \bx) \, d\bx = \calO(\epsilon^N) \,.
\]
\end{lemma}

\begin{proof}
By the Plancherel identity 
$$
\int e^{i \alpha \cdot \bx} \overline{ f(\epsilon \bx)} \, d\bx = \epsilon^{-2} \int \overline{\widehat f \left(\frac{\rk}{\epsilon} \right)} \delta(\rk - \alpha) \, d\rk = \epsilon^{-2} \overline{\widehat f \left(\frac{\alpha}{\epsilon} \right)} \,.
$$
Using that 
$|\widehat f(\rk)| \leq C_N(1+|\rk|^2)^{-\frac{N}{2}}$ for all $N$, we obtain 
$$
\left| \int e^{i \alpha \cdot \bx} \overline{ f(\epsilon \bx)} \, d\bx \right| \leq C_N \epsilon^{-2} \left| \widehat f \left(\frac{\alpha}{\epsilon} \right) \right| = \calO (\epsilon^{N-2}) \,.
$$
\end{proof}
Consequently,
$\mathcal{I}_1 = \calO(\epsilon^N)$,  $\mathcal{I}_2 = \calO(\epsilon^N)$
for all $N$, and all integrals in (\ref{new-H4-general-form}) except the third one are negligible in this modulational regime.
All terms with fast oscillations essentially homogenize to zero for $\epsilon \ll 1$.

\subsection{Quartic interactions in the modulational regime}

The homogenization step above allows us to omit the first, second, fourth and fifth integrals in (\ref{new-H4-general-form}) when approximating $H_+^{(4)}$ up to order $\mathcal{O}(\epsilon^2)$.
Then using the expression (\ref{fourth-order-H-R}) with $T_-^+ = T_1 -\frac{1}{2} T_2$ after the change of variables (\ref{kj-around-k0}), we obtain 
\begin{equation}
\label{H-plus-in-U}
H_+^{(4)} = \epsilon^2 \int \left( T_1 - \frac{1}{2} T_2\right) U_1 U_2 \bar U_3 \bar U_4 \delta_{1+2-3-4} d\chi_{1234} \,, 
\end{equation}
where $U_j$ is defined by (\ref{U-in-fourier}) and $\delta_{1+2-3-4} = \frac{1}{4 \pi^2} \int e^{-i (\chi_1 + \chi_2 - \chi_3 - \chi_4) \cdot \bx} d\bx$.
\begin{proposition}
\label{fourth-order-appr-proposition}
The  term $H^{(4)}_+$ in the reduced Hamiltonian has the form
\begin{eqnarray*}
H^{(4)}_+ & = & \frac{k_0^3 \epsilon^2}{8 \pi^2} \int
\left( 1 + \frac{3 \epsilon}{2k_0} (\lambda_2+\lambda_3) - 
\frac{\epsilon (\lambda_1-\lambda_3)^2}{k_0|\chi_1 - \chi_3|}
\right)
U_1 U_2 \bar U_3 \bar U_4 \delta_{1+2-3-4} d\chi_{1234}
+ \calO(\epsilon^4) \,.
\end{eqnarray*}
\end{proposition}
In view of (\ref{H-plus-in-U}), the proof is based on the following  lemmas.

\begin{lemma}
\label{T1-term-lemma}
Under the modulational Ansatz (\ref{kj-around-k0}), 
the coefficient $T_1$ in  Proposition  \ref{T1}  becomes 
\begin{equation}
\label{T-1-expansion-integral}
\begin{aligned}
\int T_1 U_1 U_2 \bar U_3 \bar U_4 \delta_{1+2-3-4} &d\chi_{1234} = 
~ \frac{k_0^3}{16 \pi^2} \int
U_1 U_2\bar U_3 \bar U_4 \delta_{1+2-3-4} d\chi_{1234}\\[5pt]
& + \frac{3 k_0^2 \epsilon}{32 \pi^2} \int (\lambda_2 + \lambda_3) U_1 U_2\bar U_3 \bar U_4 \delta_{1+2-3-4} d\chi_{1234}
 + \calO(\epsilon^2) \,.
\end{aligned}
\end{equation}
\end{lemma}

\begin{proof}
Using that $|{\rm k}_j| = k_0 + \epsilon \lambda_j + \calO (\epsilon^2)$, 
\begin{equation}
\label{T-1-expansion}
\begin{aligned}
T_1= & ~ \frac{k_0^2}{16 \pi^2} \Big( k_0 + \epsilon \frac{\lambda_1+\lambda_4+5(\lambda_2+\lambda_3)}{4} \Big) \\[5pt]
& + 
\frac{k_0^2\epsilon }{32 \pi^2} (|\chi_1-\chi_3|+ |\chi_2-\chi_4| - |\chi_2-\chi_3| - |\chi_1-\chi_4|) + \calO(\epsilon^2) \,. 
\end{aligned}
\end{equation}
This expression for $T_1$  can be  simplified 
by rearranging the indices.  Since 
$$
\int |\chi_1 - \chi_3| U_1 U_2 \bar U_3 \bar U_4 \delta_{1+2-3-4} d\chi_{1234} = \int |\chi_2 - \chi_3| U_1 U_2 \bar U_3 \bar U_4 \delta_{1+2-3-4} d\chi_{1234} \,,
$$
and
$$
\int |\chi_2-\chi_4| U_1 U_2 \bar U_3 \bar U_4 \delta_{1+2-3-4} d\chi_{1234} = \int |\chi_1-\chi_4| U_1 U_2 \bar U_3 \bar U_4 \delta_{1+2-3-4} d\chi_{1234} \,,
$$
the contribution from the second line in (\ref{T-1-expansion}) to the RHS integral (\ref{T-1-expansion-integral}) is zero. The first line in (\ref{T-1-expansion}) is treated via
\begin{equation*}
\label{lemma-II-proof-14-23-symmetry}
\int (\lambda_1 + \lambda_4) U_1 U_2 \bar U_3 \bar U_4 \delta_{1+2-3-4} d\chi_{1234} = \int 
(\lambda_2 + \lambda_3) U_1 U_2 \bar U_3 \bar U_4 \delta_{1+2-3-4} d\chi_{1234} \,,
\end{equation*}
leading to the desired result.
\end{proof}

\begin{lemma}
\label{I-term-lemma}
Let ${\rm I}$, ${\rm II}$, ${\rm III}$ be given as in Proposition \ref{lemma-T2-coeff}.
Assuming (\ref{kj-around-k0}), we have
$$
\begin{aligned}
&\int {\rm I} \, U_1 U_2 \bar U_3 \bar U_4 \delta_{1+2-3-4} d\chi_{1234} =  ~ \frac{k_0^3}{16 \pi^2} (\sqrt{2}-1) \int 
U_1 U_2 \bar U_3 \bar U_4 \delta_{1+2-3-4} d\chi_{1234}\\[5pt]
&\qquad + \frac{3 k_0^2 \epsilon}{32 \pi^2} (\sqrt{2}-1) \int (\lambda_2 + \lambda_3)
U_1 U_2 \bar U_3 \bar U_4 \delta_{1+2-3-4} d\chi_{1234} + \calO(\epsilon^2) \,,
\end{aligned}
$$
\begin{equation*}
\int {\rm II} \, U_1 U_2 \bar U_3 \bar U_4 \delta_{1+2-3-4} d\chi_{1234}  = 
\frac{k_0^2 \epsilon}{4 \pi^2} \int 
\frac{(\lambda_1-\lambda_3)^2}{|\chi_1 - \chi_3|} U_1 U_2 \bar U_3 \bar U_4 \delta_{1+2-3-4} d\chi_{1234} + \calO(\epsilon^2) \,,
\end{equation*}
$$
\begin{aligned}
&\int {\rm III} \, U_1 U_2 \bar U_3 \bar U_4 \delta_{1+2-3-4} d\chi_{1234} = -\int \frac{k_0^3}{16 \pi^2} (\sqrt{2}+1)
U_1 U_2 \bar U_3 \bar U_4 \delta_{1+2-3-4} d\chi_{1234}\\[5pt]
&\qquad - \frac{3 k_0^2 \epsilon}{32 \pi^2} (\sqrt{2}+1) \int (\lambda_2 + \lambda_3)
U_1 U_2 \bar U_3 \bar U_4 \delta_{1+2-3-4} d\chi_{1234} + \calO(\epsilon^2) \,.
\end{aligned}
$$

\end{lemma}

\begin{proof}
Under the assumption (\ref{kj-around-k0}), 
$\ell_{{\rm k}_1+{\rm k}_2}^{ -{\rm k}_1} = \calO (\epsilon^2)$, and so are
$\ell_{{\rm k}_1+{\rm k}_2}^{ -{\rm k}_2}, \ell_{{\rm k}_3+{\rm k}_4}^{-{\rm k}_3}$, $\ell_{{\rm k}_3+{\rm k}_4}^{-{\rm k}_4}$.
Expanding the remaining terms in (\ref{I-term-new}), 
we get 
$$
{\rm I} = \frac{k_0^3}{16 \pi^2} (\sqrt{2}-1) + \frac{3 k_0^2 \epsilon}{32 \pi^2} (\sqrt{2}-1) (\lambda_1+\lambda_2+\lambda_3+\lambda_4) + \calO(\epsilon^2) \,.
$$
The  identities for  ${\rm II}$ and ${\rm III}$ are obtained similarly.
\end{proof}

\section{Hamiltonian Dysthe equation}

The third-order normal form transformation eliminates all cubic terms from the Hamiltonian $H$. We find that,  in the modulational 
regime (\ref{kj-around-k0}), the reduced Hamiltonian is  
\begin{equation}
\label{reduced-H-for-dysthe}
\begin{aligned}
H & = H^{(2)} + H_+^{(4)} = \int \omega (\rk_0 + \epsilon \chi) |U_\chi|^2 d\chi \\
&  + 
\frac{k_0^3 \epsilon^2}{8 \pi^2} \int
\left( 1 + \frac{3 \epsilon}{2k_0} (\lambda_2+\lambda_3) - 
\frac{\epsilon (\lambda_1-\lambda_3)^2}{k_0|\chi_1 - \chi_3|}
\right)
U_1 U_2 \bar U_3 \bar U_4 \delta_{1+2-3-4} d\chi_{1234} + \calO(\epsilon^4) \,.
\end{aligned}
\end{equation}

\subsection{Derivation of the Dysthe equation}

For the purpose of returning to variables in the physical space, we introduce
\begin{equation*}
\label{z-u-relation-physical}
\begin{aligned}
z(\bx) & = \frac{1}{2\pi} \int z(\rk) e^{i\rk \cdot \bx} d\rk =  \frac{\epsilon}{2\pi} \int U(\chi) e^{i k_0 x} 
e^{i \chi \cdot \epsilon \bx} d\chi = \epsilon \, u(\bX) e^{ik_0 x} \,, 
\end{aligned}
\end{equation*}
where $u$ is the inverse Fourier transform of $U$ depending on the long spatial scale $\bX = (X, Y) = \epsilon \, \bx$, hence
\begin{equation*}
\label{new-variables-u}
\begin{pmatrix}
u \\ \bar u
\end{pmatrix} = P_2 \begin{pmatrix}
z \\ \bar z  
\end{pmatrix} = \epsilon^{-1} \begin{pmatrix}
e^{-i k_0 x} & 0 \\
0 & e^{i k_0 x}
\end{pmatrix} \begin{pmatrix}
z \\ \bar z 
\end{pmatrix} \,.
\end{equation*}
From  (\ref{z-physical-ham-system}), the evolution equations for $(u, \bar u)$ are 
\begin{equation}
\label{u-physical-ham-system}
\partial_t \begin{pmatrix}
u \\ \bar u
\end{pmatrix} = J_2 \begin{pmatrix}
\partial_u H \\ \partial_{\bar u} H
\end{pmatrix} = \begin{pmatrix}
0 & -i \\ i & 0
\end{pmatrix} \begin{pmatrix}
\partial_u H \\ \partial_{\bar u} H
\end{pmatrix} \,,
\end{equation}
where $J_2 = \epsilon^2 P_2 J_1 P_2^*$. 

We now derive a Dysthe equation for the slowly varying wave envelope $u$ 
which, by construction,  has the property of being Hamiltonian.
The starting point is  (\ref{u-physical-ham-system}) with $H$ being the truncated Hamiltonian  (\ref{reduced-H-for-dysthe}).
In the $(u, \bar u)$ variables, 
the quadratic part $H^{(2)} $ becomes
\begin{equation*}
\label{quadratic-H-in-u0}
\begin{aligned}
H^{(2)} = & \int \bar{u} \, \omega(\rk_0+\epsilon  D) u \, d\bX \,.
\end{aligned}
\end{equation*}
Taylor expanding the linear dispersion relation leads to
\begin{equation*}
\label{quadratic-H-in-u}
\begin{aligned}
&H^{(2)}  = \int \omega_0  |u|^2 d\bX \\
& + \frac{\omega_0}{2} \int
\bar u \left(  \frac{\epsilon }{k_0} D_X  - \frac{ \epsilon^2}{4 k_0^2} D^2_X +\frac{ \epsilon^2 }{2k_0^2} D^2_Y + i \frac{\epsilon^3}{8 k_0^3} D_X^3 
- \frac{3\epsilon^3}{4 k_0^3} D_X D_Y^2 \right) u \, d\bX + \calO(\epsilon^4) \,,
\end{aligned}
\end{equation*}
where $\omega_0 = \sqrt{g k_0}$ and $D = (D_X, D_Y) = -i \, (\partial_X, \partial_Y)$.
Turning to the quartic term of the truncated Hamiltonian, we obtain the following result.
\begin{lemma}
\label{H4-in-u}
In the $(u, \bar u)$ variables, 
the quartic  term $H_+^{(4)}$ in  (\ref{reduced-H-for-dysthe}) 
 is
\begin{equation}
\label{quartic-H-in-u}
\begin{aligned}
H_+^{(4)} = & ~ \frac{1}{2} \int \left( \epsilon^2 k_0^3 |u|^4 + 3\epsilon^3 k_0^2 |u|^2 {\rm Im} (\bar u \partial_X u) + \epsilon^3 k_0^2 |u|^2 ~\partial_X^2 |D|^{-1} |u|^2 \right) d\bX + \calO(\epsilon^4) \,.
\end{aligned}
\end{equation}
\end{lemma}

\begin{proof}
Using that $U$ is the Fourier transform of $u$, we have 
$$
U_1 U_2 \bar U_3 \bar U_4 = \frac{1}{16 \pi^4} \int u_1 u_2 \bar u_3 \bar u_4 e^{-i(\chi_1 \cdot \bX_1 + \chi_2 \cdot \bX_3 - \chi_3 \cdot \bX_3 - \chi_4 \cdot \bX_4)} d\bX_{1234} \,,
$$
and using the definition of $\delta_{1+2-3-4}$ yields
\begin{align} \label{eq5-11}
 \int U_1 U_2 \bar U_3 \bar U_4 \delta_{1+2-3-4} d\chi_{1234} 
& = 4 \pi^2 \int |u|^4 d\bX \,.
\end{align}
Similarly, 
$$
\begin{aligned}
& \int \lambda_2 U_1 U_2 \bar U_3 \bar U_4 \delta_{1+2-3-4} d\chi_{1234}  = - 4 i \pi^2 \int |u|^2 \bar u \, \partial_X u \, d\bX \,,
\end{aligned}
$$
$$
\begin{aligned}
& \int \lambda_3 U_1 U_2 \bar U_3 \bar U_4 \delta_{1+2-3-4} d\chi_{1234}  = 4 i  \pi^2 \int |u|^2 u \, \overline{ \partial_X u} \, d\bX \,,
\end{aligned}
$$
and
\begin{equation} \label{eq5-12}
\int (\lambda_2 + \lambda_3) U_1 U_2 \bar U_3 \bar U_4 \delta_{1+2-3-4} d\chi_{1234} = 8 \pi^2 \int |u|^2 {\rm Im}(\bar u \partial_X u) \, d\bX \,.
\end{equation}
The remaining term of $H_+^{(4)}$ amounts to 
\begin{equation} \label{eq5-13}
\int \frac{(\lambda_1 - \lambda_3)^2}{|\chi_1-\chi_3|} U_1 U_2 \bar U_3 \bar U_4 \delta_{1+2-3-4} d\chi_{1234} = - 4 \pi^2 \int |u|^2 \partial_X^2|D|^{-1} |u|^2 \, d\bX \,.
\end{equation}
Combining (\ref{eq5-11})--(\ref{eq5-13}),
we get the quartic term  $H_+^{(4)}$  given in   
(\ref{quartic-H-in-u}).
\end{proof}

The resulting reduced Hamiltonian takes the form
\begin{eqnarray} \label{Hamil2}
H & = & \int \omega_0 |u|^2 + \varepsilon \frac{\omega_0}{2 k_0} \operatorname{Im}(\overline u \partial_X u)
- \varepsilon^2 \frac{\omega_0}{8 k_0^2} |\partial_X u|^2 + \varepsilon^2 \frac{\omega_0}{4 k_0^2} |\partial_Y u|^2 
+ \varepsilon^2 \frac{k_0^3}{2} |u|^4 \nonumber \\
& & \quad + \varepsilon^3 \frac{\omega_0}{16 k_0^3} \operatorname{Im}\!\left[(\overline{\partial_X u})(\partial_X^2 u)\right]
- \varepsilon^3 \frac{3 \omega_0}{8 k_0^3} \operatorname{Im}\!\left[(\overline{\partial_X u})(\partial_Y^2 u)\right]
+ \varepsilon^3 \frac{3 k_0^2}{2} |u|^2 \operatorname{Im}(\overline u \partial_X u) \nonumber \\
& & \quad + \varepsilon^3 \frac{k_0^2}{2} |u|^2 \partial_X^2 |D|^{-1} |u|^2 d\bX +  \calO(\varepsilon^4) \,.
\end{eqnarray}
It follows from (\ref{u-physical-ham-system}) that the evolution equation for $u$ up to order $\calO(\epsilon^3)$ is
\begin{eqnarray}
\label{dysthe-eqn}
i \, \partial_t u & = & \partial_{\bar u} H \nonumber \\ 
& = & \omega_0 u - i \epsilon \frac{\omega_0}{2 k_0} \partial_X u + \epsilon^2 \frac{\omega_0}{8k_0^2} \partial_X^2 u - \epsilon^2 \frac{\omega_0}{4k_0^2} \partial_Y^2 u + \epsilon^2 k_0^3 |u|^2 u \nonumber \\ 
& & + i \epsilon^3 \frac{\omega_0}{16 k_0^3} \partial_X^3 u - i \epsilon^3 \frac{3 \omega_0}{8 k_0^3} \partial_X \partial_Y^2 u - 3i \epsilon^3 k_0^2 |u|^2 \partial_X u + \epsilon^3 k_0^2 u ~\partial_X^2 |D|^{-1} |u|^2 \,,
\end{eqnarray}
which is  a Hamiltonian version of Dysthe's equation for three-dimensional gravity waves on deep water.  
It describes modulated waves moving in the positive $x$-direction at group velocity
$\partial_{k_x} \omega(\rk_0) = \frac{\omega_0}{2 k_0}$ as shown by the advection term. The  nonlocal term  $u \, \partial_X^2 |D|^{-1} |u|^2$ is a signature
of the Dysthe equation. It reflects the presence of the wave-induced mean flow as in the classical derivation using the method of multiple scales.
It reduces to $-u \, |D_X| |u|^2$ in the two-dimensional case.

\begin{remark}
It has been suggested in \cite{TKDV00} that keeping the linear dispersion relation exact, rather than expanding it in powers of $\epsilon$ as done above, may
provide an overall better approximation. In this Hamiltonian setting, the resulting envelope equation would take the form 
\begin{equation*}
\label{dysthe-eqn2}
i \, \partial_t u = \omega(\rk_0+\epsilon D ) u  +\epsilon^2 k_0^3 |u|^2 u 
 - 3i \epsilon^3 k_0^2 |u|^2 \partial_X u + \epsilon^3 k_0^2 u ~\partial_X^2 |D|^{-1} |u|^2 \,.
\end{equation*}
\end{remark}

\subsection{Moving reference frame}

We can further simplify the Hamiltonian \eqref{Hamil2} by subtracting a multiple of the wave action
$
M = \int |u|^2 \, d\bX 
$
together with a multiple of the  impulse
\[
I = \int \eta \, \partial_\bx \xi \, d\bx = \int {\rm k}_0 |u|^2 
+ \epsilon \operatorname{Im}(\overline u \partial_{\bX} u) \, d\bX \,,
\]
yielding
\begin{eqnarray*}
\widehat H & = & H - \partial_{{\rm k}} \omega({\rm k}_0) \cdot I 
- \Big[ \omega_0 - {\rm k}_0 \cdot \partial_{{\rm k}} \omega({\rm k}_0) \Big] M \,. 
\end{eqnarray*}
Since  $M$ and $I$ are conserved with respect to the flow of $\widehat H$,  
they Poisson commute with $H$ \cite{CGS21b}. This transformation preserves the symplectic structure $J_2$
and the resulting simplification of \eqref{dysthe-eqn}  reads, after introduction of the slow time  $\tau = \epsilon^2 t$,
\begin{equation*}
\label{dysthe-eqn-slow-time}
\begin{aligned}
i \, \partial_\tau u & =  \frac{\omega_0}{8k_0^2} \partial_X^2 u -  \frac{\omega_0}{4k_0^2} \partial_Y^2 u +  k_0^3 |u|^2 u \\[3pt]
& \quad + i \epsilon \frac{\omega_0}{16 k_0^3} \partial_X^3 u - i \epsilon \frac{3 \omega_0}{8 k_0^3} \partial_X \partial_Y^2 u - 3i \epsilon k_0^2 |u|^2 \partial_X u+ \epsilon k_0^2 u ~\partial_X^2 |D|^{-1} |u|^2 \,.
\end{aligned}
\end{equation*}

\section{Reconstruction of the free surface}

\subsection{Approximation of auxiliary Hamiltonian $K^{(3)}$}

Reconstruction of the free surface from the wave envelope requires solving the auxiliary Hamiltonian system \eqref{Eq:Hamiltonflow-Fourier}.
Its numerical computation is costly in general because this involves evaluating multiple multi-dimensional integrals
which are not convolutions and thus cannot be calculated by the FFT.
As an alternative, we propose a simplified version that can be solved efficiently by exploiting the fact that wave propagation
is primarily in the $x$-direction according to the modulational Ansatz \eqref{kj-around-k0}.

Introducing $\K = (\K_x, \K_y)$ with $k_x = \K_x$, $k_y = \epsilon \K_y$ such that
\begin{equation}
\label{etan-xin-new}
\etan_\kappa = \etan (\K_x, \K_y) := \eta (\K_x, \epsilon \K_y) \,, \quad
\xin_\kappa = \xin (\K_x, \K_y) := \xi (\K_x, \epsilon \K_y) \,,
\end{equation}
the coefficients inside the integrals (\ref{K3-fourier})  
can be expanded. In particular, 
$$ 
  d_{123} =  d_{123}^x + \epsilon^2 d_{123}^R + \calO(\epsilon^4) \,,
$$
where
\begin{equation*}
\label{exp-coeffs}
\begin{aligned}
 d_{123}^x &:= g^2 \left( 
\K_{1x}^2 + \K_{2x}^2 + \K_{3x}^2 - 
2|\K_{1x}||\K_{2x}| - 2|\K_{1x}||\K_{3x}| -
2|\K_{2x}||\K_{3x}| \right) \,, \\[5pt]
 \frac{d_{123}^R}{g^2}  &:= 
\frac{\K_{1y}^2}{|\K_{1x}|}
(|\K_{1x}|-|\K_{2x}|-|\K_{3x}|) +
\frac{\K_{2y}^2}{|\K_{2x}|}
(|\K_{2x}|-|\K_{1x}|-|\K_{3x}|)  \nonumber\\
& \quad + 
\frac{\K_{3y}^2}{|\K_{3x}|}
(|\K_{3x}|-|\K_{1x}|-|\K_{2x}|) \,. \\[5pt]
\end{aligned}
\end{equation*}
The contribution $d_{123}^x$ identifies to the denominator in the two-dimensional case. It reduces to $-4 g^2 |\K_{1x}||\K_{3x}|$ in the region
where $\K_{1x} + \K_{2x}+ \K_{3x} =0$ and  $\K_{1x} \K_{3x }>0$.

The above computations allow us to derive the expansion of $K'^{(3)}$ up to order $\calO(\epsilon^6)$. The leading-order term 
identifies to the formula for $K^{(3)}$  in two dimensions (see \cite{CS16} Theorem 3.8), while the correction term is much more complicated
as shown below.

\begin{proposition}
\label{proposition-new-K3-expansion}
The expansion of $K^{'(3)}$ in  the regime    (\ref{etan-xin-new}) is 
\begin{equation}
\label{K3-fourier-new}
\begin{aligned}
K^{'(3)} (\etan, \xin) = & - \frac{\epsilon^2}{4\pi}
\int \sgn (\K_{1x}) \sgn (\K_{2x}) |\K_{3x}| 
  \etan_1 \etan_2  \xin_3  
\delta_{123} d\K_{123} \\[5pt]
& + \frac{\epsilon^4}{2\pi} 
\int \left( 
R_{123}   \etan_1 \etan_2 \xin_3
+ Q_{123}   \xin_1   \xin_2 \xin_3 \right)
\delta_{123} d\K_{123} + \calO (\epsilon^6) \,,
\end{aligned}
\end{equation}
where  $R_{123}$ and $Q_{123}$ are given by
\begin{eqnarray*}
\label{r-123-term-proposition}
R_{123} & = &
\frac{\K_{1y}^2 |\K_{2x}|}{4\K_{1x}^2} 
\big( \sgn (\K_{1x}) \sgn (\K_{2x}) - 
\sgn (\K_{1x}) \sgn (\K_{3x}) \big) \\
& & - \frac{\K_{1y}^2}{4|\K_{1x}|} 
\big( 1 + \sgn (\K_{1x}) \sgn (\K_{3x}) \big) 
 -  \frac{\K_{1y} \K_{2y}}{2 |\K_{1x}|} 
\big(1 - \sgn (\K_{2x}) \sgn (\K_{3x}) \big) \,, \\
Q_{123} & = & - \frac{1}{8g} 
\Big( \frac{\K_{1y}^2 |\K_{3x}|}{|\K_{1x}|} + \frac{\K_{3y}^2 |\K_{1x}|}{|\K_{3x}|} - 
2 \K_{1y} \K_{3y} \sgn (\K_{1x}) \sgn (\K_{3x}) \Big) \,.
\label{q-123-term-proposition}
\end{eqnarray*}
\end{proposition}

\subsection{Reconstruction procedure}

Retaining only the leading-order term in (\ref{K3-fourier-new}) for $K'^{(3)}$, the new coordinates are obtained as solutions of 
\begin{equation*}\label{NFT-flow-fourier-new-eta}
\partial_s \etan_{\K} =  
\epsilon^{-1} \partial_{\xin_{-\K}} K^{'(3)}
 =  - \frac{\epsilon}{4\pi} 
\int \sgn (\K_{1x}) \sgn (\K_{2x}) |\K_{x}| 
\etan_1   \etan_2  
\delta_{1+2-\K} d\K_{12} \,,
\end{equation*}
\begin{equation*}\label{NFT-flow-fourier-new-xi}
\partial_s \xin_{\K} =  
- \epsilon^{-1} \partial_{\etan_{-\K}} K^{'(3)} =
 - \frac{\epsilon}{2\pi} 
\int \sgn (\K_{x}) \sgn (\K_{1x}) |\K_{2x}| 
  \etan_1   \xin_2  
\delta_{1+2-\K} d\K_{12} \,.
\end{equation*}
Back to the physical space, 
\begin{equation*}
\label{inverse-FT-relation-new}
(\eta (\bx), \xi (\bx)) = \frac{\epsilon}{2 \pi} 
\int ( \etan_{\kappa}, 
\xin_{\kappa})
e^{i \kappa \cdot (x, \epsilon y)}
d \kappa \,,
\end{equation*}
satisfy the evolution equations
\[
\partial_s \eta(\bx) = - \frac{1}{2} |D_x| \left( \sgn(D_x) \eta \right)^2 \,, \quad
\partial_s \xi(\bx) =  - \sgn (D_x) \big( (\sgn(D_x) \eta) (|D_x|\xi) \big) \,.
\]
Via the new variables 
$
\etah := -i \, \sgn (D_x) \eta
$
and
$
\xih := -i \, \sgn (D_x) \xi 
$
involving the Hilbert $x$-transform,
this system simplifies to 
\begin{equation} \label{NFT-flow-physical-etah}
\partial_s \etah(\bx) = -  (\partial_x \etah) \etah \,, \quad 
\partial_s \xih(\bx) = - \etah \partial_x \xih \,,
\end{equation}
which preserves the canonical Hamiltonian structure as in the two-dimensional case \cite{CGS21,CS16}.
The equation for $\etah$ is the Burgers equation while the equation for $\xih$ is its linearization along the Burgers flow.
Integrating (\ref{NFT-flow-physical-etah}) up to $s = -1$, with initial conditions at $s=0$ being the transformed variables, 
provides a reconstruction of the actual free surface.

\begin{figure}
\centering
\subfloat{\includegraphics[width=.5\linewidth]{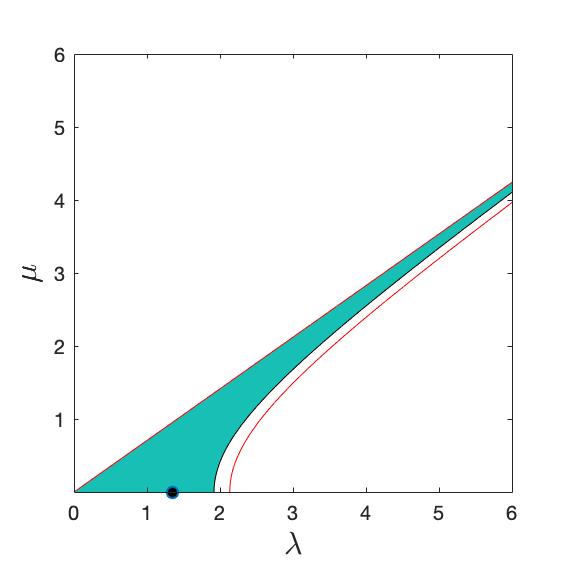}}
\hfill
\subfloat{\includegraphics[width=.5\linewidth]{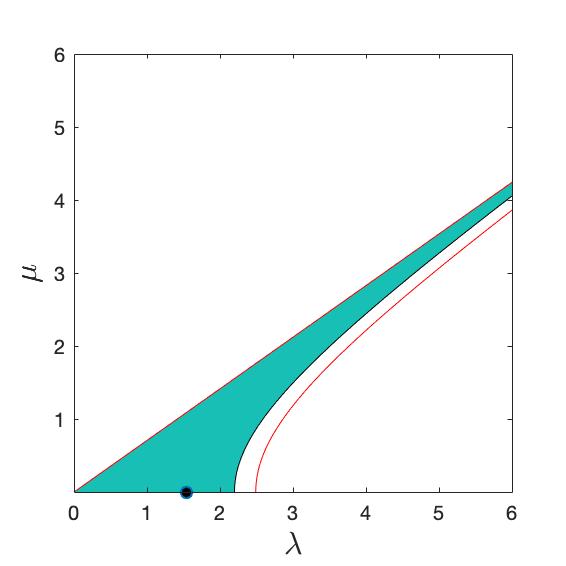}}
\caption{Region of modulational instability (shaded area) for \eqref{dysthe-eqn}. 
The corresponding region for the NLS equation is the extension represented by a red curve. 
The black dot represents the mode associated with maximum growth.
Left panel: $(B_0,k_0) = (0.003,10)$. Right panel: $(B_0,k_0) = (0.0035,10)$.}
\label{BF_inst}
\end{figure}

\section{Numerical results}

We present numerical results to illustrate the performance of our Hamiltonian approach
in the context of modulational instability of Stokes waves.
We first present a theoretical analysis and then show some numerical simulations in comparison to other models.

\subsection{Stability of Stokes waves}

Equation \eqref{dysthe-eqn} admits a uniform solution of the form
\[
u_0(t)  = B_0 e^{-i (\omega_0 + \varepsilon^2 k_0^3 B_0^2) t} \,, 
\]
representing a progressive Stokes wave ($B_0$ being a positive real constant).
Such a solution is known to be linearly unstable with respect to sideband perturbations,
which is referred to as modulational or Benjamin--Feir (BF) instability.
We provide a version of this analysis based on \eqref{dysthe-eqn} for the three-dimensional problem.

Inserting a perturbation of the form
\[
u(\bX,t) = u_0(t) \big[ 1 + B(\bX,t) \big] \,,
\]
where
\[
B(\bX,t) = B_1 e^{\Omega t + i (\lambda X + \mu Y)} 
+ B_2 e^{\overline \Omega t - i (\lambda X + \mu Y)} \,,
\]
and $B_1, B_2$ are complex coefficients, we find that the condition 
$\operatorname{Re}(\Omega) \neq 0$ for instability implies
\begin{equation} \label{BFCond}
\left( \frac{\lambda^2}{2} - \mu^2 \right) \left[ 2 k_0^2 B_0^2 \left( k_0 - \varepsilon \frac{\lambda^2}{\sqrt{\lambda^2 + \mu^2}} \right)
- \frac{\omega_0}{4 k_0^2} \left( \frac{\lambda^2}{2} - \mu^2 \right) \right] > 0 \,.
\end{equation}
This is a tedious but straightforward calculation for which we skip the details.
We refer the reader to \cite{D79,TKDV00} for similar calculations.

Figure \ref{BF_inst} shows instability regions in the $(\lambda,\mu)$-plane as predicted by condition \eqref{BFCond}
for $k_0 = 10$ and two different amplitudes $B_0 = 0.003$ and $0.0035$.
Hereafter, all the variables are rescaled to absorb $\varepsilon$ back into their definition,
and all the equations are non-dimensionalized so that $g = 1$.
These two plots correspond to wave steepnesses $\varepsilon = k_0 A_0 = 0.075$ and $0.088$ respectively,
based on the relationship
\[
B_0 = A_0 \sqrt[4]{\frac{g}{4k_0}} \,,
\]
between the envelope amplitude $B_0$ and the surface amplitude $A_0$ according to \eqref{zeta-variable-fourier}.
In both cases, we see that the instability region is unbounded,
extending in the form of a narrow strip to higher wavenumbers from the origin.
Maximum growth (strongest instability) is achieved at $\mu = 0$ and $\lambda \simeq 1.5$
which is a longitudinal long-wave mode.
The instability region for the NLS equation, if $\varepsilon = 0$ in \eqref{BFCond}, turns out to be larger
and its extent relative to the Dysthe prediction is represented by a red curve in Fig. \ref{BF_inst}.

\subsection{Simulations and comparisons}

To validate our Hamiltonian approach, we test it against the full water wave system \eqref{ww-hamiltonian-equation}
which is given more explicitly by
\begin{equation} \label{HamMotionEqn}
\partial_t \eta = G(\eta) \xi \,, \quad
\partial_t \xi = -g \eta - \frac{1}{2} |\partial_\bx \xi|^2
+ \frac{\big( G(\eta) \xi + \partial_\bx \eta \cdot \partial_\bx \xi \big)^2}{2(1 + |\partial_\bx \eta|^2)} \,.
\end{equation}
We also compare our model predictions to solutions of the classical (non-Hamiltonian) Dysthe equation 
\begin{eqnarray} \label{Trulsen}
i \, \partial_t A & = & -\frac{i \omega_0}{2k_0} \partial_x A + \frac{\omega_0}{8 k_0^2} \partial_x^2 A 
- \frac{\omega_0}{4 k_0^2} \partial_y^2 A
+ \frac{1}{2} \omega_0 k_0^2 |A|^2 A + \frac{i \omega_0}{16 k_0^3} \partial_x^3 A \nonumber \\
& & - \frac{3 i \omega_0}{8 k_0^3} \partial_x \partial_y^2 A - \frac{3 i}{2} \omega_0 k_0 |A|^2 \partial_x A
- \frac{i}{4} \omega_0 k_0 A^2 \partial_x \overline A + k_0 A \, \partial_x \Phi \,,
\end{eqnarray}
where 
\[
\Phi = \frac{1}{2} \omega_0 \partial_x |D|^{-1} |A|^2 \,, \quad \partial_x \Phi = \frac{1}{2} \omega_0 \partial_x^2 |D|^{-1} |A|^2 \,,
\] 
denote contributions from the wave-induced mean flow.
In this formulation, the surface elevation and velocity potential are reconstructed perturbatively 
in terms of the Stokes expansion
\begin{eqnarray} \label{Stokes}
\eta(\bx,t) & = & \frac{1}{2\omega_0} \partial_x \Phi + \operatorname{Re}\!\left[ A e^{i \theta} 
+ \frac{1}{2} (k_0 A^2 - i A \partial_x A) e^{2 i \theta} + \frac{3}{8} k_0^2 A^3 e^{3 i \theta} \right] 
+ \ldots \nonumber  \\
\varphi(\bx,z,t) & = & \Phi + \operatorname{Re}\!\left[ \Big( -\frac{i \omega_0}{k_0} A 
+ \frac{\omega_0}{2 k_0^2} \partial_x A + \frac{3 i \omega_0}{8 k_0^3} \partial_x^2 A \right. \nonumber \\
& & \left. \qquad \qquad - \frac{i \omega_0}{4 k_0^3} \partial_y^2 A + \frac{i}{8} \omega_0 k_0 |A|^2 A \Big) 
e^{i \theta} e^{k_0 z} \right] + \ldots 
\end{eqnarray}
up to third harmonics, as typically reported in the literature \cite{GT11,T06}.
The phase function is given by $\theta = k_0 x - \omega_0 t$.
As mentioned earlier, Eqs. \eqref{Trulsen} and \eqref{Stokes} are expressed in terms of unscaled variables
for the purposes of this comparative study.

The full equations \eqref{HamMotionEqn} are solved numerically following a high-order spectral approach \cite{CS93}.
They are discretized in space by a pseudo-spectral method based on the FFT.
The computational domain spans $0 \le x \le L_x$, $0 \le y \le L_y$ with doubly periodic boundary conditions
and is divided into a regular mesh of $N_x \times N_y$ collocation points.
The DNO is computed via its series expansion \eqref{series} but, by analyticity, a small number $m$ of terms is sufficient
to achieve highly accurate results. The number $m = 4$ is selected based on previous extensive tests \cite{GN07,XG09}.
Time integration of \eqref{HamMotionEqn} is carried out in the Fourier space 
so that the linear terms can be solved exactly by the integrating factor technique.
The nonlinear terms are integrated in time by using a 4th-order Runge--Kutta scheme with constant step $\Delta t$.
The same numerical methods are applied to the envelope equations \eqref{dysthe-eqn} and \eqref{Trulsen},
as well as to their reconstruction formulas, with the same resolutions in space and time.
In particular, Burgers equation \eqref{NFT-flow-physical-etah} is integrated in $s$ by using the same step size $\Delta s = \Delta t$.
As noted in our previous work on the two-dimensional problem \cite{CGS21},
the additional cost of solving this relatively simple equation is insignificant.

Initial conditions of the form
\[
u(\bx,0) = B_0 \big[ 1 + 0.1 \cos(\lambda x) \cos(\mu y) \big] \,, \quad 
A(\bx,0) = A_0 \big[ 1 + 0.1 \cos(\lambda x) \cos(\mu y) \big] \,,
\]
are specified to define a perturbed Stokes wave for \eqref{dysthe-eqn} and \eqref{Trulsen} respectively.
Accordingly, it is important to ensure that appropriate initial conditions are prescribed for \eqref{HamMotionEqn}:
Eqs. \eqref{NFT-flow-physical-etah} with $u(\bx,0)$ (resp. Eqs. \eqref{Stokes} with $A(\bx,0)$) are used 
when the full equations are compared to predictions from \eqref{dysthe-eqn} (resp. \eqref{Trulsen}).

\begin{figure}
\centering
\subfloat{\includegraphics[width=.5\linewidth]{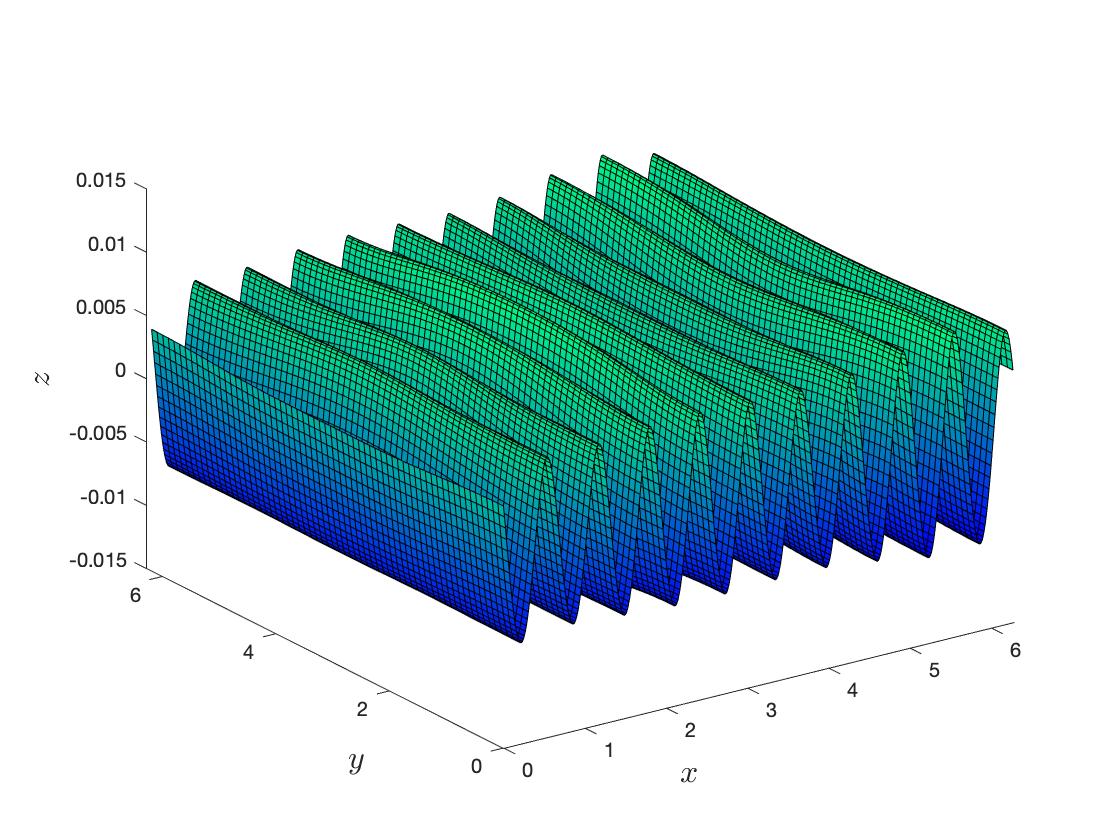}}
\hfill
\subfloat{\includegraphics[width=.5\linewidth]{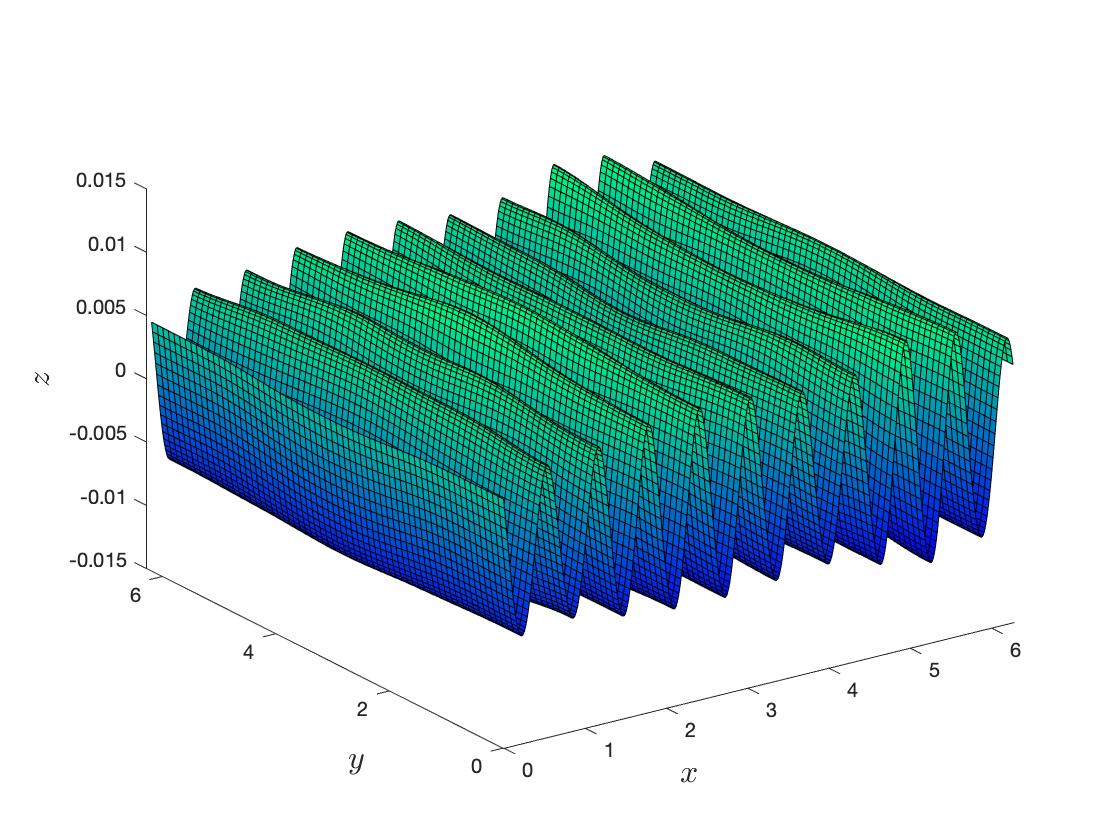}}
\caption{Surface elevation $\eta$ at $t = 2500$ for $B_0 = 0.003$, $k_0 = 10$ and $(\lambda,\mu) = (1,1)$.
Left panel: Hamiltonian Dysthe equation. Right panel: fully nonlinear equations.}
\label{surf_a0003_t2500}
\end{figure}

\begin{figure}
\centering
\subfloat[]{\includegraphics[width=.5\linewidth]{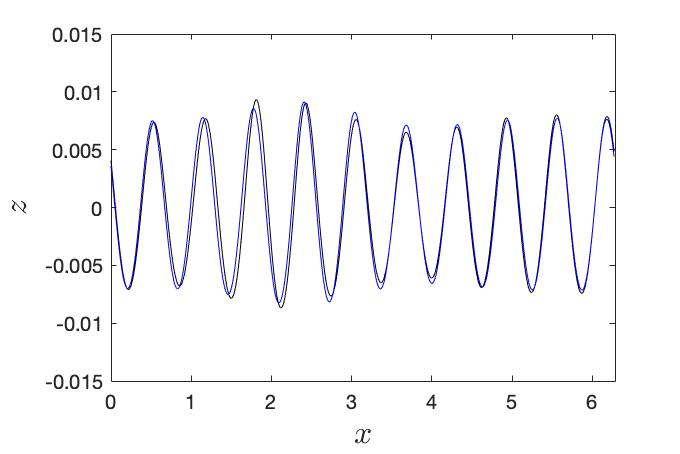}}
\hfill
\subfloat[]{\includegraphics[width=.5\linewidth]{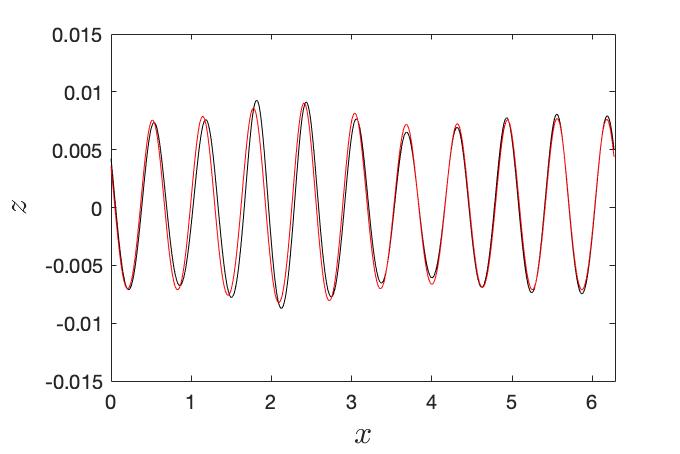}}
\caption{Comparison on $\eta$ between the fully and weakly nonlinear solutions in the cross-section $y = L_y/2$ 
at $t = 2500$ for $B_0 = 0.003$, $k_0 = 10$ and $(\lambda,\mu) = (1,1)$.
Left panel: Hamiltonian Dysthe equation in blue. Right panel: classical Dysthe equation in red.
The black curve represents the fully nonlinear solution.}
\label{comp_y12_a0003_t2500}
\end{figure}

The following tests focus on the two cases considered in the previous stability analysis.
The initial wave parameters are $k_0 = 10$, $B_0 = 0.003$ or $0.0035$,
and $(\lambda,\mu) = (1,1)$ so that the initial condition is a Stokes wave under 
three-dimensional long-wave (i.e. sideband) perturbations.
The computational domain is taken to be of size $L_x \times L_y = 2\pi \times 2\pi$.
The spatial and temporal resolutions are set to $\Delta x = 0.012$ ($N_x = 512$), $\Delta y =  0.098$ ($N_y = 64$)
and $\Delta t = 0.005$.
This difference in resolution between $x$ and $y$ reflects the choice of $x$ as the preferred direction of wave propagation.

\begin{figure}
\centering
\subfloat{\includegraphics[width=.5\linewidth]{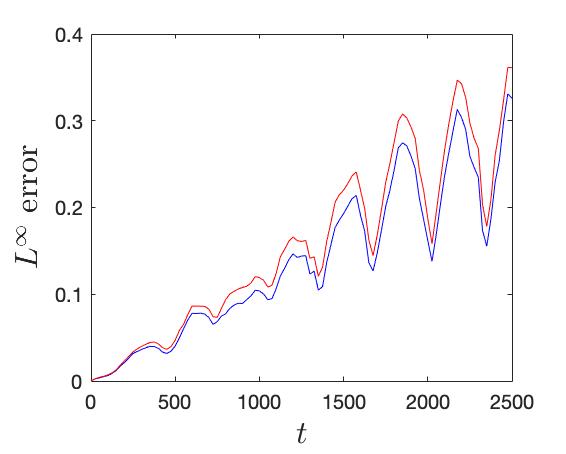}}
\hfill
\subfloat{\includegraphics[width=.5\linewidth]{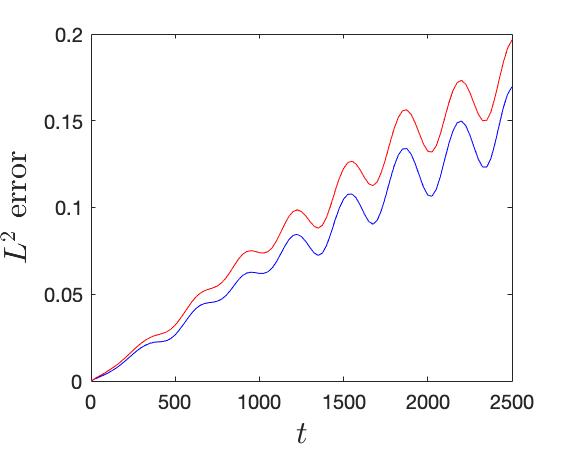}}
\caption{Relative errors on $\eta$ versus time between the fully and weakly nonlinear solutions  
for $B_0 = 0.003$, $k_0 = 10$ and $(\lambda,\mu) = (1,1)$.
The blue curve represents the Hamiltonian Dysthe equation while the red curve represents the classical Dysthe equation.
Left panel: $L^\infty$ error. Right panel: $L^2$ error.}
\label{error_a0003_t2500}
\end{figure}

For the first case with small initial data ($B_0 = 0.003$), 
we examine the wave dynamics over a long time up to $t = 2500 = \calO(\varepsilon^{-3})$
given the initial steepness $\varepsilon = 0.075$.
This corresponds to the time scale over which the Dysthe approximation with such initial data is supposed to be valid.
Figure \ref{surf_a0003_t2500} shows the full surface elevation $\eta$ at $t = 2500$ 
as predicted from \eqref{dysthe-eqn} and \eqref{HamMotionEqn}. 
A more direct comparison between these two solutions is reported in Fig. \ref{comp_y12_a0003_t2500}(a)
along the cross-section $y = L_y/2$.
A similar test for the classical Dysthe equation \eqref{Trulsen} is depicted in Fig. \ref{comp_y12_a0003_t2500}(b).
Under such a mild disturbance, effects of modulational instability are not felt yet.
In both cases, the wave profiles remain close to their initial configuration and, as a result, 
both plots look quite similar.
A more quantitative assessment is provided in Fig. \ref{error_a0003_t2500} 
which displays the time evolution of the relative $L^\infty$ and $L^2$ errors 
\begin{equation} \label{errors}
\frac{\| \eta_f - \eta_w \|_\infty}{\| \eta_f \|_\infty} \,, \quad \frac{\| \eta_f - \eta_w \|_2}{\| \eta_f \|_2} \,,
\end{equation}
on $\eta$ between the fully ($\eta_f$) and weakly ($\eta_w$) nonlinear solutions.
The low values confirm that both Dysthe models perform very well in this case,
with the predictions from \eqref{dysthe-eqn} being slightly better than those from \eqref{Trulsen}.
This is consistent with results for the two-dimensional problem
and is expected considering that the surface reconstruction for \eqref{dysthe-eqn} 
is a non-perturbative procedure as opposed to the perturbative calculation for \eqref{Trulsen}.

\begin{figure}
\centering
\subfloat[]{\includegraphics[width=.5\linewidth]{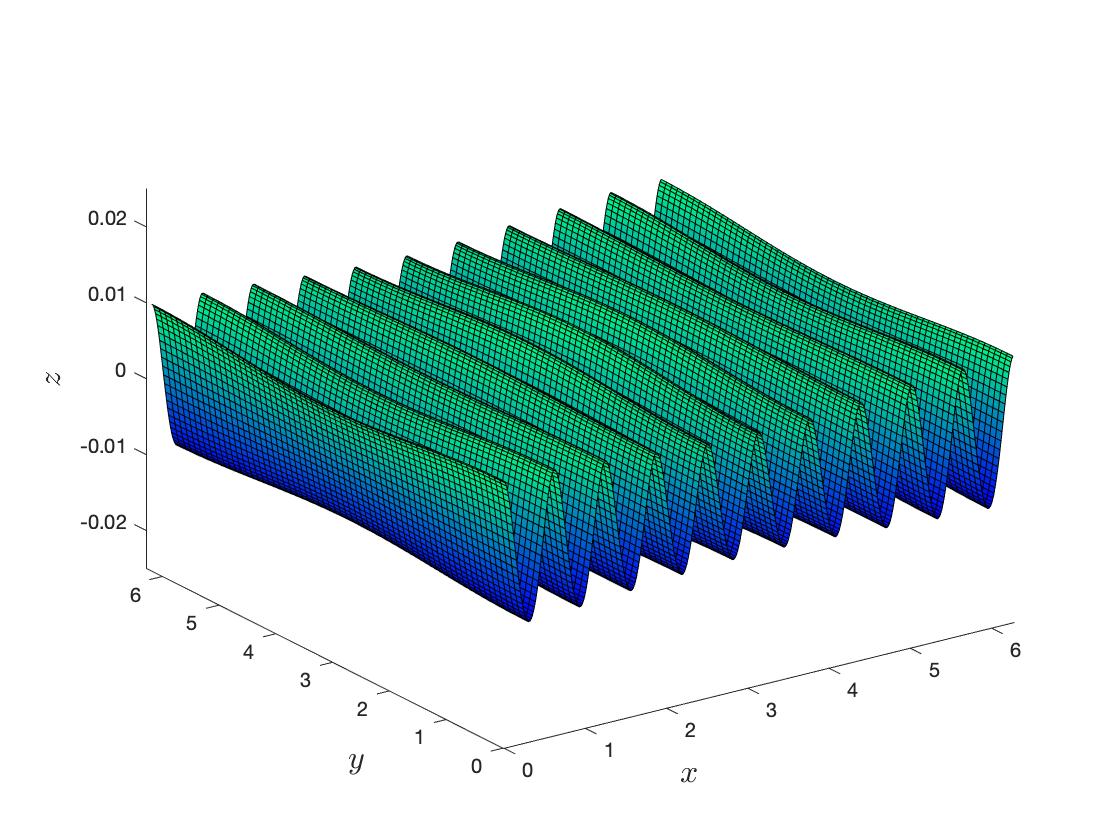}}
\hfill
\subfloat[]{\includegraphics[width=.5\linewidth]{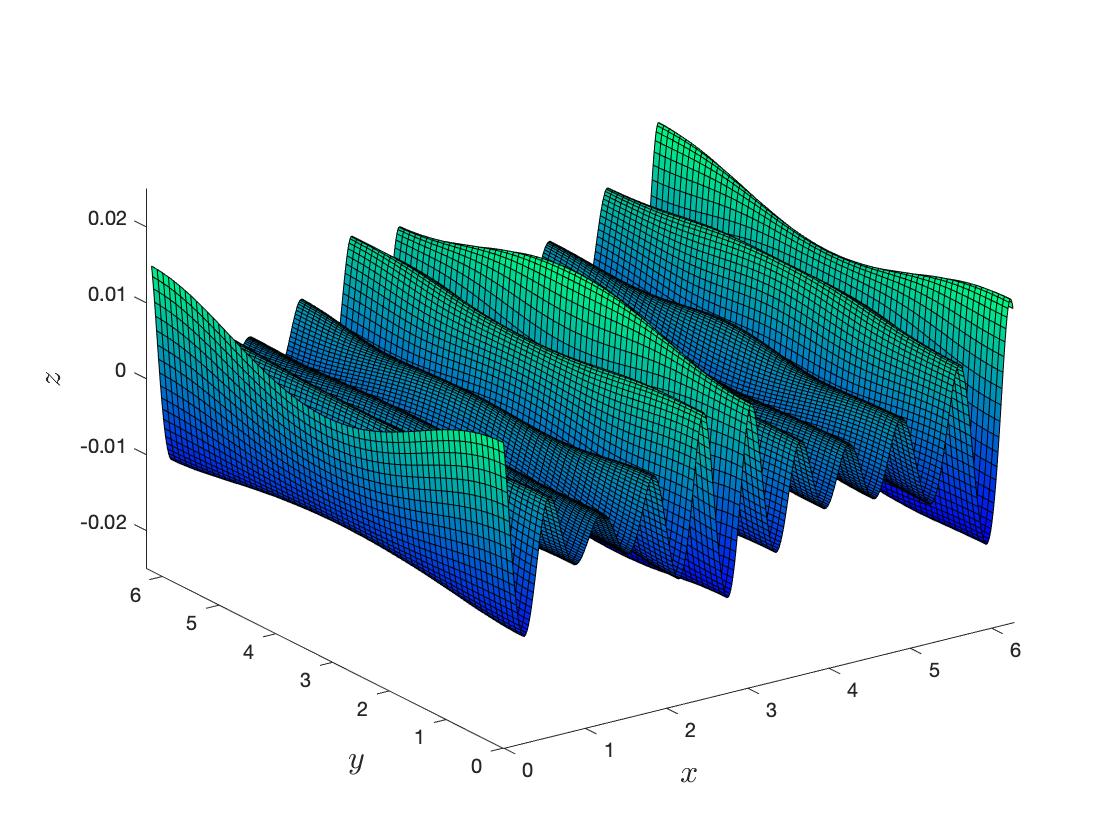}}
\hfill
\subfloat[]{\includegraphics[width=.5\linewidth]{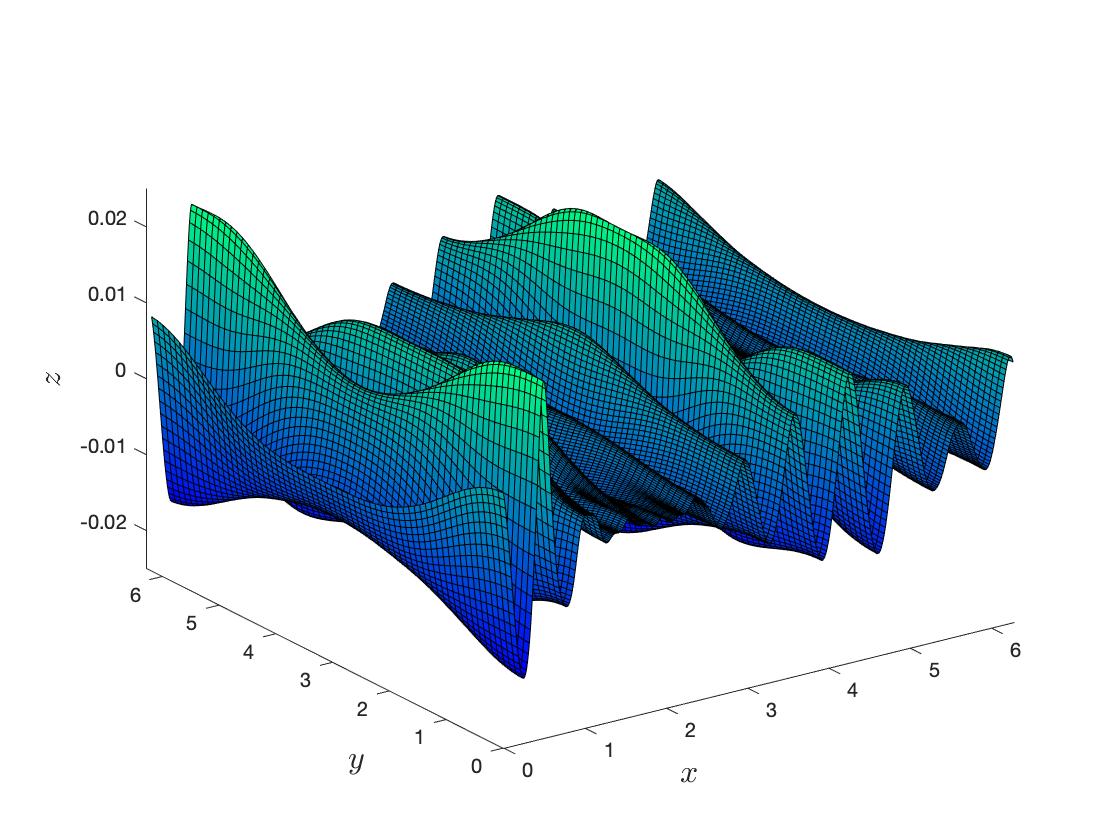}}
\hfill
\subfloat[]{\includegraphics[width=.5\linewidth]{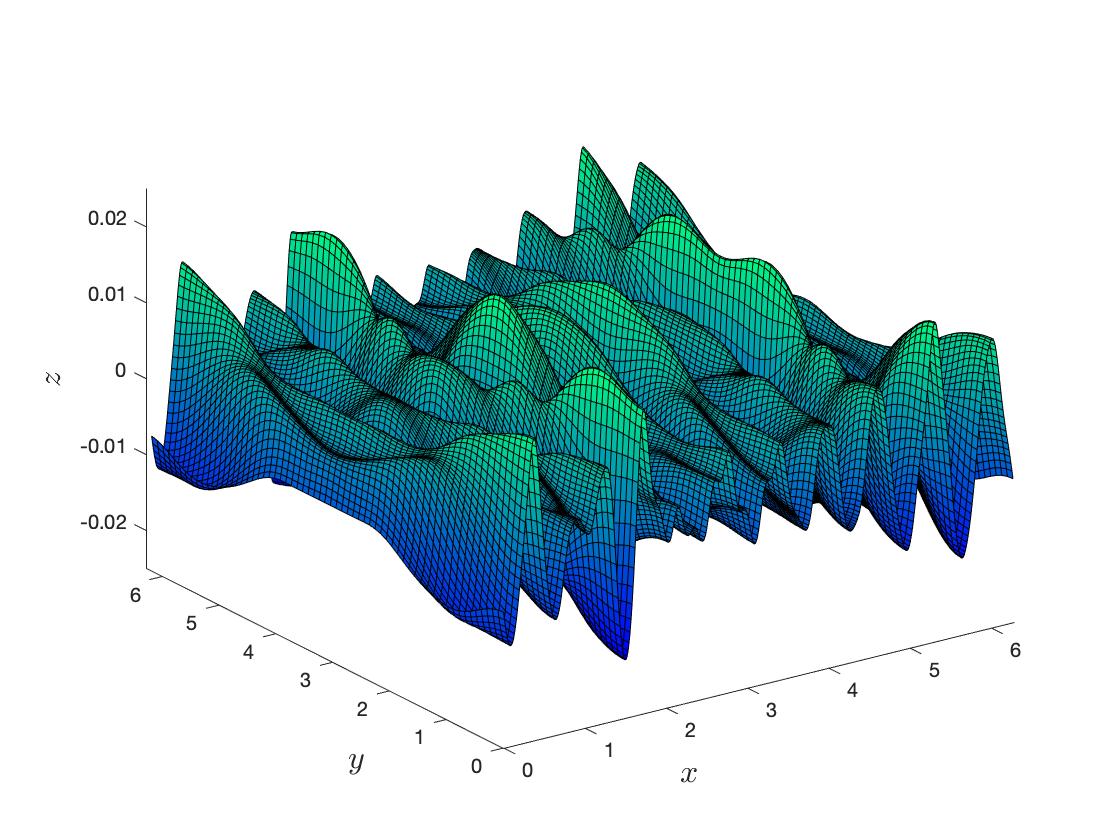}}
\hfill
\subfloat[]{\includegraphics[width=.5\linewidth]{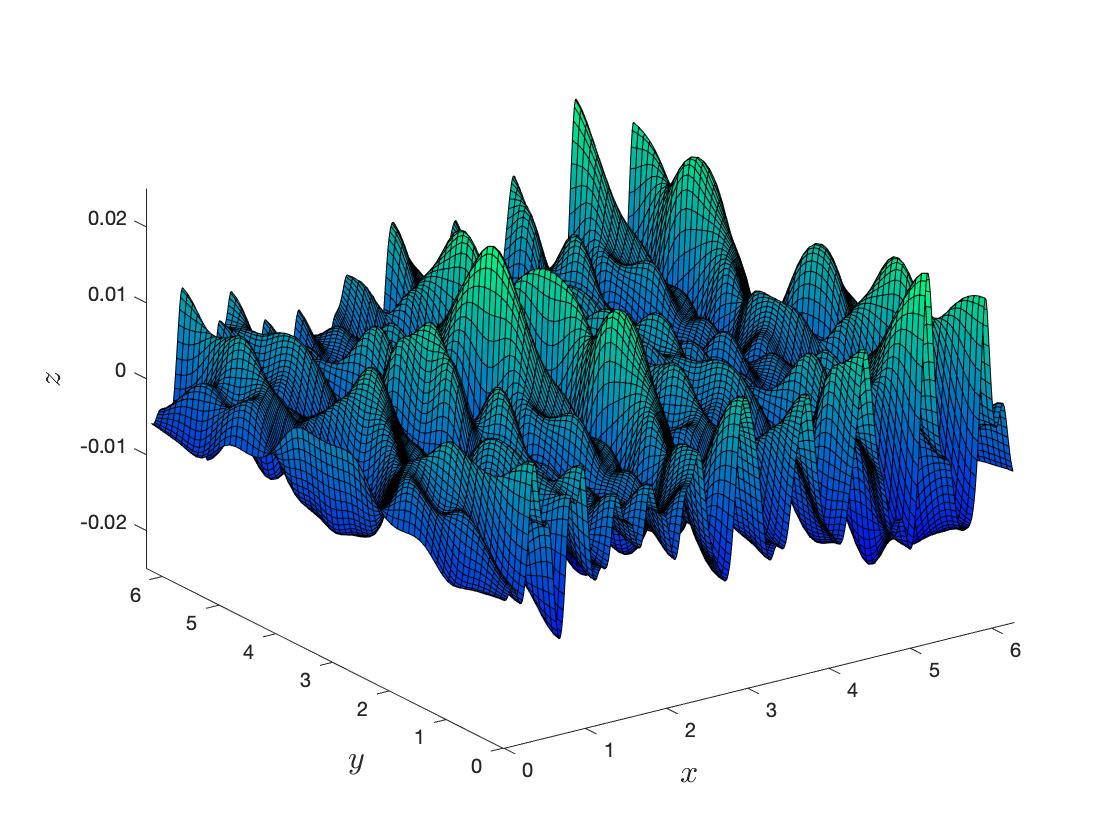}}
\hfill
\subfloat[]{\includegraphics[width=.5\linewidth]{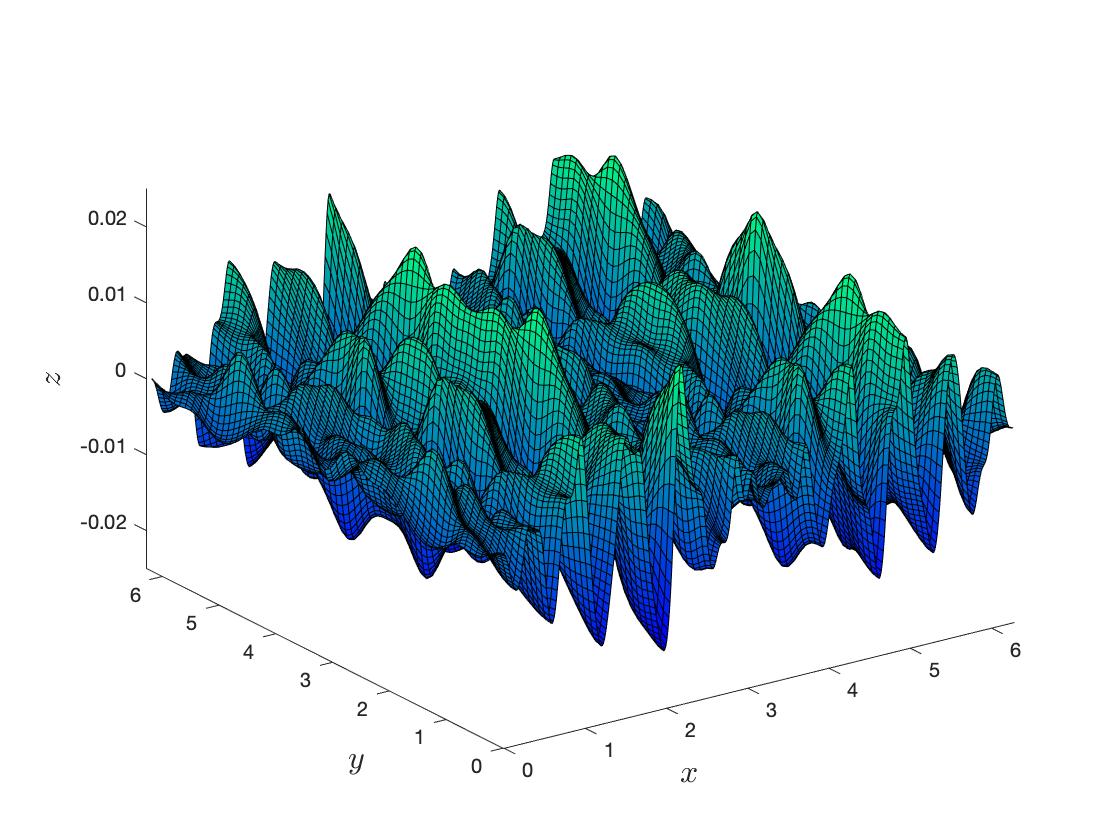}}
\caption{Surface elevation $\eta$ as predicted from the Hamiltonian Dysthe equation
at (a) $t = 0$, (b) $t = 540$, (c) $t = 680$, (d) $t = 810$, (e) $t = 930$, (f) $t = 1000$
for $B_0 = 0.0035$, $k_0 = 10$ and $(\lambda,\mu) = (1,1)$.}
\label{hamdys3_surf_a00035}
\end{figure}

The second case with larger initial data ($B_0 = 0.0035$, $\varepsilon = 0.088$)
is more prone to modulational instability.
Snapshots of the full surface elevation up to $t = 1000$ are presented in Fig. \ref{hamdys3_surf_a00035}.
As expected, the Stokes wave becomes unstable under the incipient development of the longitudinal sideband mode 
$\lambda = 1$ (and near mode $\lambda = 2$) around $t = 680$. 
However, unlike the two-dimensional situation where a quasi-recurrent cycle of modulation-demodulation
typically takes place over a long time \cite{CGS21}, these perturbations quickly trigger the excitation of higher sideband modes
in both horizontal directions, leading to the emergence of an irregular short-crested wave field.
This phenomenon is observed in both our weakly and fully nonlinear simulations,
which is consistent with results from previous numerical studies \cite{LM87,MMMSY81}.
In particular, computations by McLean et al. \cite{MMMSY81} showed that three-dimensional instabilities 
become dominant when the wave steepness is sufficiently large.
In the present modulational regime, the gradual excitation of higher modes during wave evolution
may be anticipated based on the stability analysis from Sec. 7.1,
which reveals that the instability region for three-dimensional perturbed Stokes waves is not confined to the first few modes 
but extends over a wide range in the $(\lambda,\mu)$-plane.
The energy initially contained in low sidebands can leak to higher unstable modes, 
similar to the scenario reported by Martin and Yuen \cite{MY80} in the context of the NLS equation.
The resulting choppy sea appearance as illustrated in Fig. \ref{hamdys3_surf_a00035} at $t = 930$ and $1000$
is an indication of the limited range of applicability of the narrowband approximation
in this three-dimensional case, even for moderate initial steepnesses.

\begin{figure}
\centering
\subfloat{\includegraphics[width=.33\linewidth]{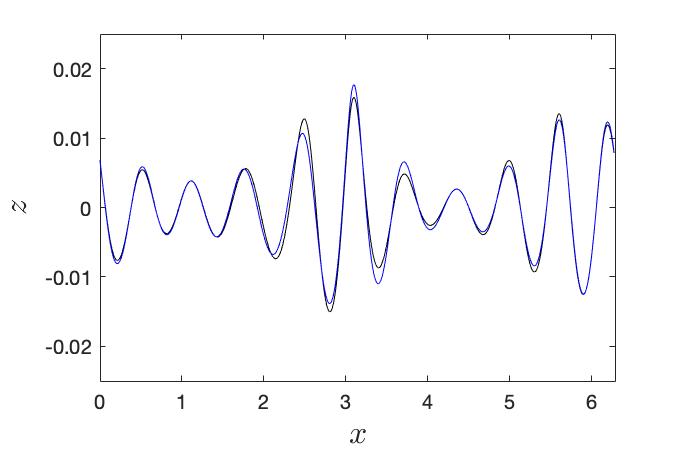}}
\hfill
\subfloat{\includegraphics[width=.33\linewidth]{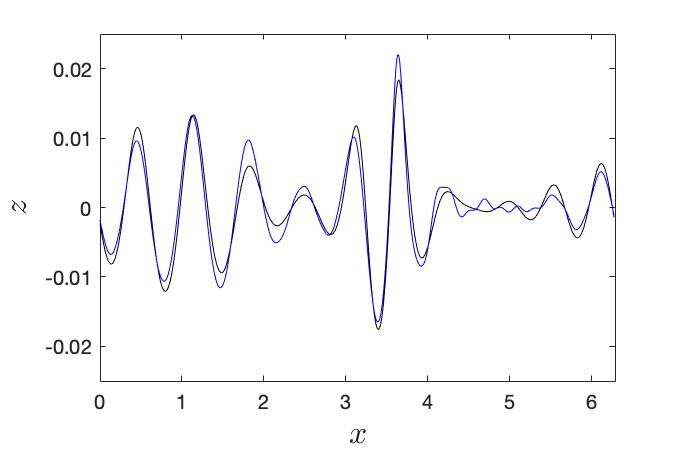}}
\hfill
\subfloat{\includegraphics[width=.33\linewidth]{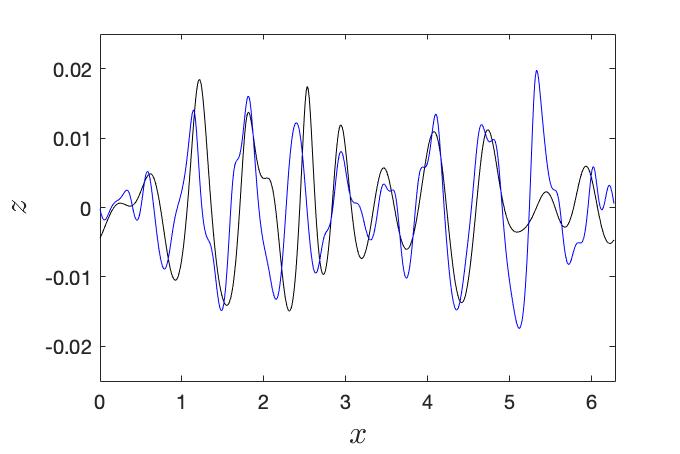}}
\hfill
\subfloat{\includegraphics[width=.33\linewidth]{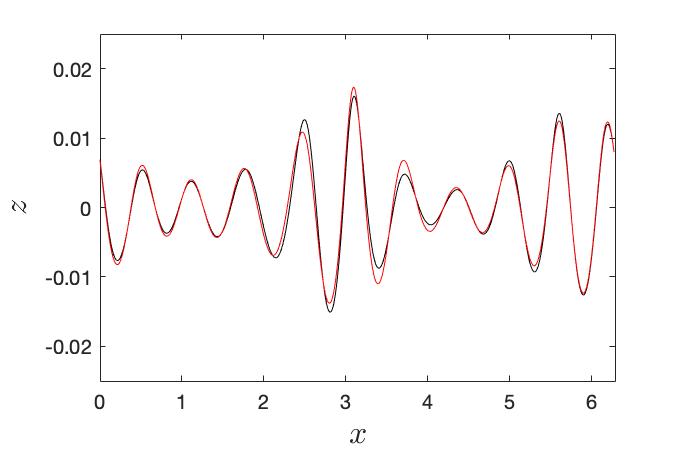}}
\hfill
\subfloat{\includegraphics[width=.33\linewidth]{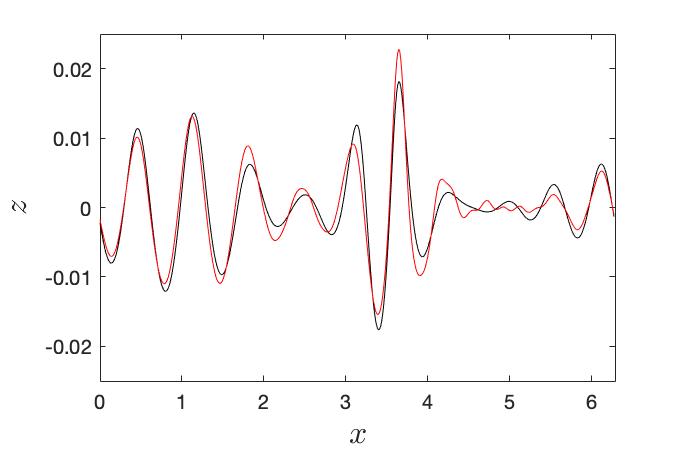}}
\hfill
\subfloat{\includegraphics[width=.33\linewidth]{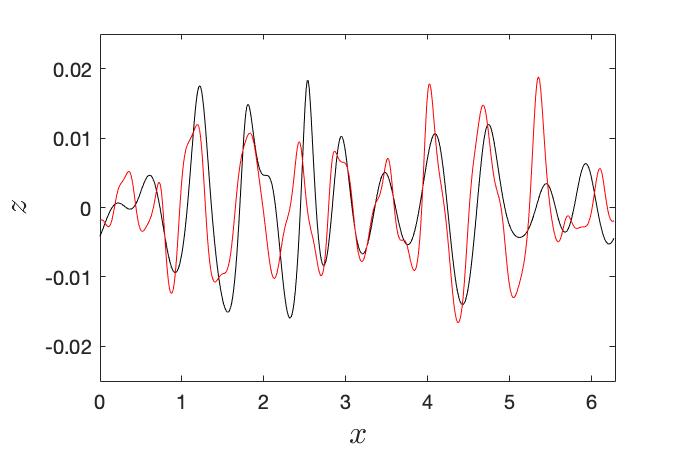}}
\caption{Comparison on $\eta$ between the fully and weakly nonlinear solutions in the cross-section $y = L_y/2$ 
at $t = 540$, $680$, $1000$ (from left to right) for $B_0 = 0.0035$, $k_0 = 10$ and $(\lambda,\mu) = (1,1)$.
Upper panels: Hamiltonian Dysthe equation in blue. Lower panels: classical Dysthe equation in red.
The black curve represents the fully nonlinear solution.}
\label{comp_y12_a00035}
\end{figure}

\begin{figure}
\centering
\subfloat{\includegraphics[width=.33\linewidth]{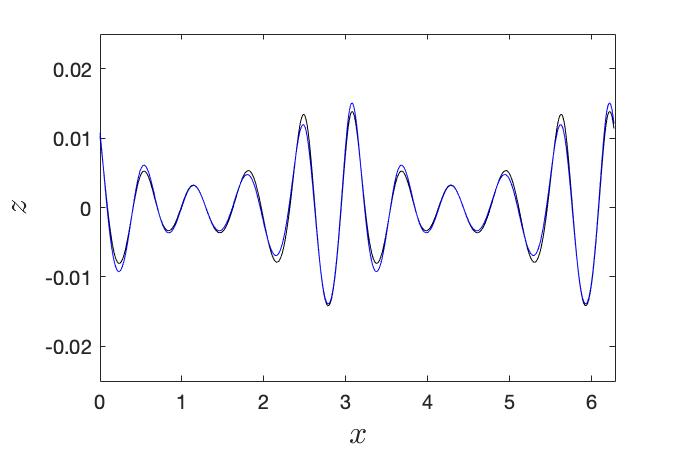}}
\hfill
\subfloat{\includegraphics[width=.33\linewidth]{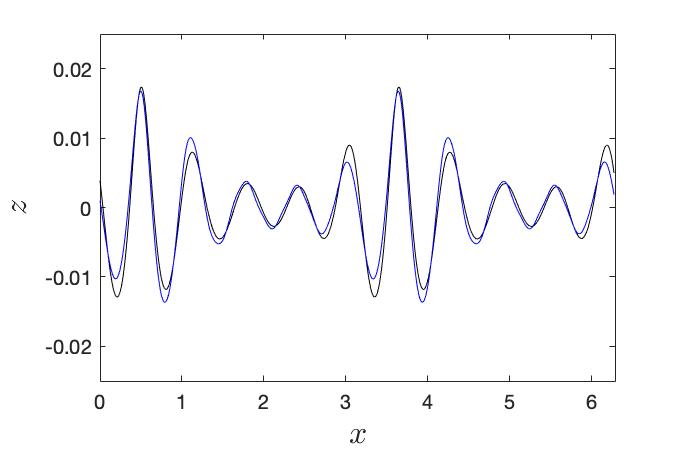}}
\hfill
\subfloat{\includegraphics[width=.33\linewidth]{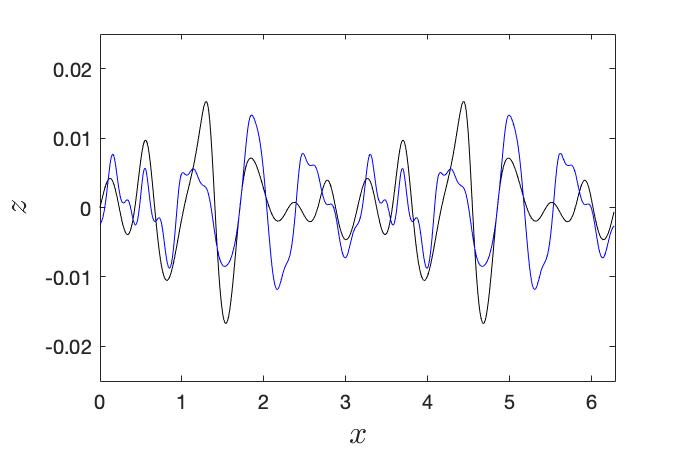}}
\hfill
\subfloat{\includegraphics[width=.33\linewidth]{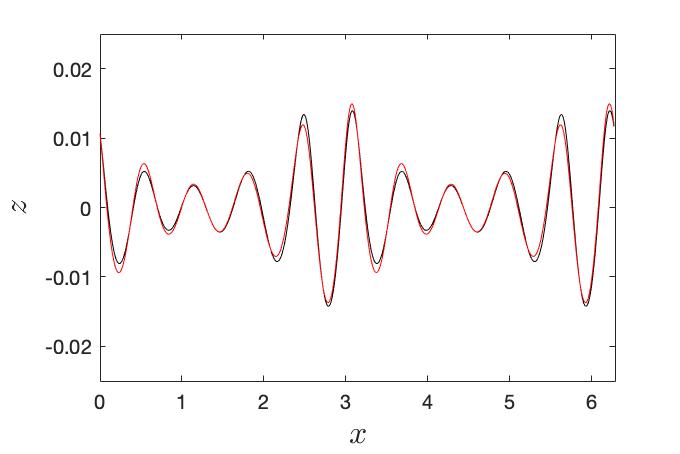}}
\hfill
\subfloat{\includegraphics[width=.33\linewidth]{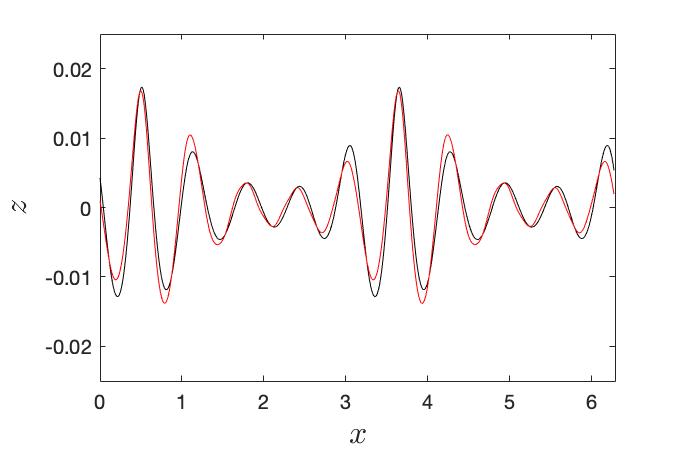}}
\hfill
\subfloat{\includegraphics[width=.33\linewidth]{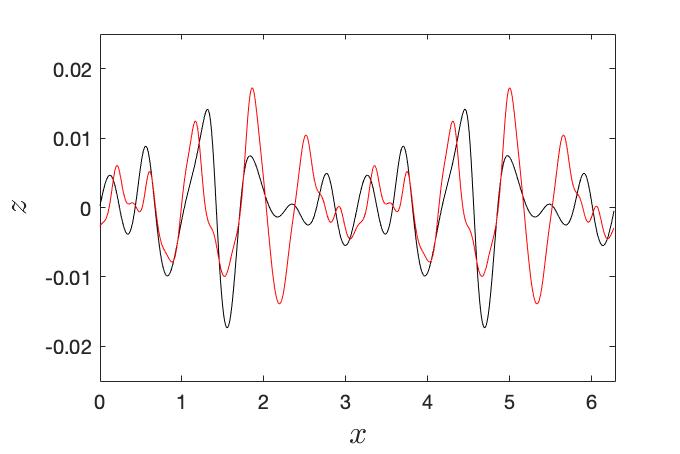}}
\caption{Comparison on $\eta$ between the fully and weakly nonlinear solutions in the cross-section $y = 3L_y/4$ 
at $t = 540$, $680$, $1000$ (from left to right) for $B_0 = 0.0035$, $k_0 = 10$ and $(\lambda,\mu) = (1,1)$.
Upper panels: Hamiltonian Dysthe equation in blue. Lower panels: classical Dysthe equation in red.
The black curve represents the fully nonlinear solution.}
\label{comp_y34_a00035}
\end{figure}

Comparison of \eqref{dysthe-eqn} and \eqref{Trulsen} with \eqref{HamMotionEqn}
is given in Figs. \ref{comp_y12_a00035} and \ref{comp_y34_a00035} 
along the cross-sections $y = L_y/2$ and $3 L_y/4$ respectively.
Snapshots of $\eta$ at $t = 540$ (early stage of BF instability), $t = 680$ (around the time of BF maximum growth)
and $t = 1000$ (short-crested wave field) are presented,
where we can clearly see the development of the longitudinal mode $\lambda = 1$ and near mode $\lambda = 2$
reaching a maximum amplitude of about $0.02$.
As suspected earlier, discrepancies between the weakly and fully nonlinear solutions are quite pronounced at $t = 1000$
along both cross-sections. This is especially true for the crest and trough heights, 
while there is still good agreement on the phase overall.
Differences between the classical and Hamiltonian Dysthe solutions also become more noticeable as time goes by.

\begin{figure}
\centering
\subfloat{\includegraphics[width=.5\linewidth]{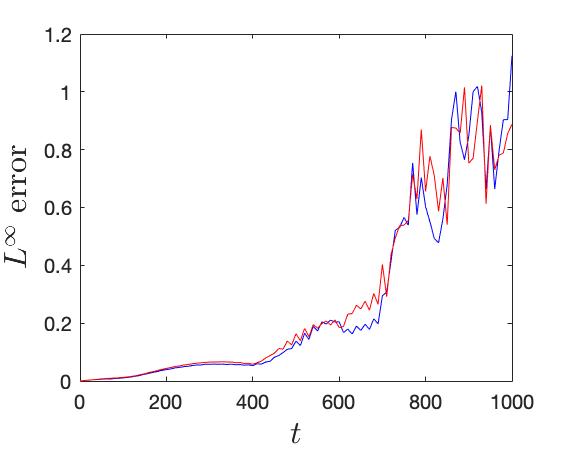}}
\hfill
\subfloat{\includegraphics[width=.5\linewidth]{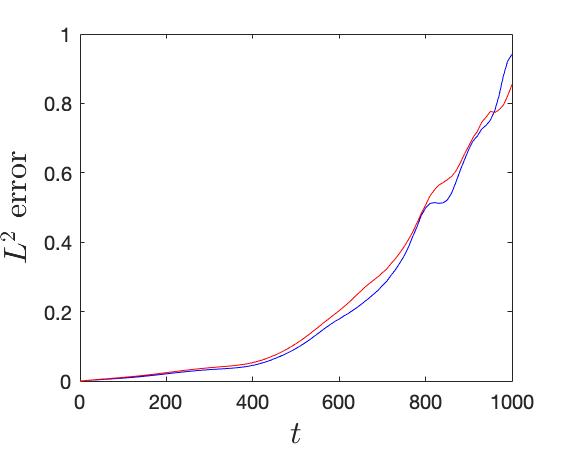}}
\caption{Relative errors on $\eta$ versus time between the fully and weakly nonlinear solutions  
for $B_0 = 0.0035$, $k_0 = 10$ and $(\lambda,\mu) = (1,1)$.
The blue curve represents the Hamiltonian Dysthe equation while the red curve represents the classical Dysthe equation.
Left panel: $L^\infty$ error. Right panel: $L^2$ error.}
\label{error_a00035_t1000}
\end{figure}

The relative $L^\infty$ and $L^2$ errors in Fig. \ref{error_a00035_t1000} again tend to slightly favor our Hamiltonian approach.
We note however that the situation seems to be reversed around $t = 960$,
with the errors for the classical Dysthe equation being lower from this point on.
Having said that, this switch occurs at a late stage of modulational instability when errors are significant
(near $80$\%) and thus very likely either weakly nonlinear model is no longer suitable,
as suggested in Figs. \ref{comp_y12_a00035} and \ref{comp_y34_a00035}.
An interpretation for this switch is that, because the Burgers equation automatically generates higher-order harmonics
of the surface wave spectrum via nonlinear interactions, it may in turn excessively amplify errors 
as the validity of the Hamiltonian Dysthe equation deteriorates over time.
From these plots, it seems that the time of validity is under $t = 1000$ which contrasts with
the expected time scale $\calO(\varepsilon^{-3}) \sim 1500$ based on the initial steepness. 
This value however is an overestimate in this case
because the wave steepness increases as a result of modulational instability.

\begin{figure}
\centering
\subfloat{\includegraphics[width=.5\linewidth]{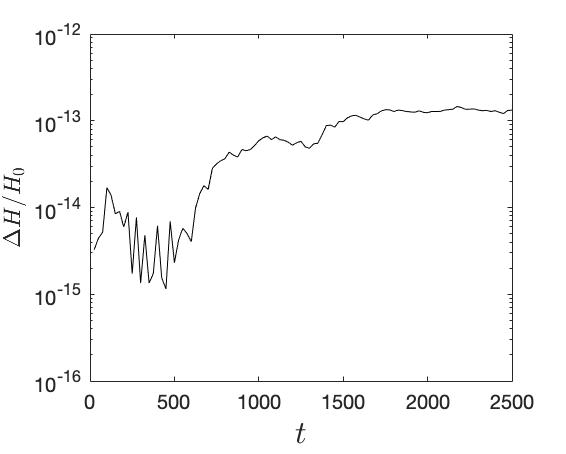}}
\hfill
\subfloat{\includegraphics[width=.5\linewidth]{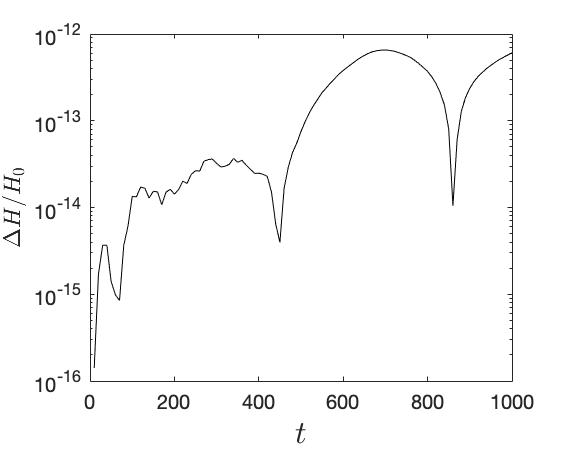}}
\caption{Relative error on $H$ versus time for the Hamiltonian Dysthe equation with $k_0 = 10$ and $(\lambda,\mu) = (1,1)$.
Left panel: $B_0 = 0.003$. Right panel: $B_0 = 0.0035$.}
\label{ener}
\end{figure}

Finally, the time evolution of the relative error
\[
\frac{\Delta H}{H_0} = \frac{|H - H_0|}{H_0} \,,
\]
on energy \eqref{Hamil2} associated with the Hamiltonian model \eqref{dysthe-eqn} is illustrated in Fig. \ref{ener} 
for $B_0 = 0.003$ and $0.0035$.
Double integrals in \eqref{Hamil2} and in the $L^2$ norm \eqref{errors} are computed via the double trapezoidal rule 
over the periodic square $[0,2\pi] \times [0,2\pi]$.
The reference value $H_0$ denotes the initial value of \eqref{Hamil2} at $t = 0$.
Overall, $H$ is very well conserved in both cases, despite a gradual loss of accuracy over time 
that is likely due to accumulation of numerical errors.

\section{Conclusions}

We propose a new Hamiltonian version of Dysthe's equation for the nonlinear modulation of three-dimensional gravity waves
on deep water. Starting from Zakharov's formulation of the water wave problem, we perform a change to Birkhoff normal form 
that is devoid of non-resonant triads, together with a sequence of canonical transformations, in order to obtain a reduced system.
The slowly varying wave envelope is introduced via a modulational Ansatz and the presence of multiple scales is handled
through a homogenization procedure.
The free surface is reconstructed by solving an auxiliary Hamiltonian system of differential  equations.
As a consequence, the entire solution process fits within a Hamiltonian framework.
To validate this approach, we conduct numerical simulations on the modulational instability of Stokes waves
and compare them to direct computations based on the full three-dimensional water wave system 
as well as to predictions from the classical Dysthe equation.
Various wave conditions are examined and very good agreement is obtained within the range of validity of this approximation.
In the future, we envision to extend these results to the finite-depth case for which 
a derivation of the third-order Birkhoff normal form is expected to be significantly more complicated.

\appendix

\section{From physical to Fourier variables} \label{App-0}

We rewrite  the Hamiltonian system (\ref{ww-hamiltonian-equation}) as
\begin{equation*}
\partial_t \begin{pmatrix}
\eta \\ \xi \\ \bar \eta \\ \bar\xi
\end{pmatrix} =  J_1 \, \nabla H(\eta,\xi, \bar \eta, \bar \xi)  = \begin{pmatrix}
J & {O}_{2\times 2} \\ 
{O}_{2\times 2} & J
\end{pmatrix} \begin{pmatrix}
\partial_\eta H \\ \partial_\xi H \\ \partial_{\bar \eta} H \\ \partial_{\bar \xi} H 
\end{pmatrix} \,,
\end{equation*}
where $J$ is given by (\ref{ww-hamiltonian-equation}) and ${O}_{2 \times 2}$ is the $2\times 2$ zero matrix. 
Denoting by $(\mathcal{F} \eta, \mathcal{F} \xi)$ the Fourier transforms of  $(\eta, \xi)$,  
we define the transformation $\tau : v=(\eta,\xi,\bar \eta, \bar \xi) \mapsto  w= (\overline {\mathcal{F} {\eta}}, \overline {\mathcal{F} {\xi}}, 
{\mathcal{F} {\eta}},  {\mathcal{F} {\xi}} )$
and write  $\widetilde H (w) = H(v)$. 
Applying calculus rules of transformations \cite{CGK05}, $w$ satisfies the system
\begin{equation}
\label{w-hamiltonian-appendix}
\partial_t  w = J_2 \, \nabla \widetilde H(w) \,,
\end{equation}
where $J_2= (\partial_v \tau) J_1 (\partial_v \tau)^{*}$ and $(\partial_v \tau)^*$  is the adjoint matrix operator, such that
\[
\partial_v \tau =
\begin{pmatrix}
 {O}_{2\times 2} & {F}^{-1} \\ {F} & {O}_{2\times 2} 
 \end{pmatrix} \,, \quad 
( \partial_v \tau)^* = \begin{pmatrix}
{O}_{2\times 2} & { F} \\ { F}^{-1} & {O}_{2\times 2}
\end{pmatrix} \,,
\]
with $F = \mathcal{F} I_{2 \times 2}$ and $F^{-1} = \mathcal{F}^{-1} I_{2 \times 2}$.
Computing the product of matrices, 
$$
J_2 = \begin{pmatrix}
F^{-1} J F^{-1} & O_{2\times 2}\\
O_{2\times 2} & F J F
\end{pmatrix} = \begin{pmatrix}
F^{-2} J & O_{2\times 2}\\
O_{2\times 2} & F^{2} J
\end{pmatrix} \,.
$$
Applying the matrix representation of $J_2$ to (\ref{w-hamiltonian-appendix}), we find 
\begin{equation*}
\label{iff-hamiltonian}
\partial_t \begin{pmatrix}
\eta_{\rm k} \\ \xi_{\rm k}
\end{pmatrix} = F^2 J \begin{pmatrix}
\partial_{\eta_{\rm k}} \widetilde H \\
\partial_{\xi_{\rm k}} \widetilde H
\end{pmatrix} \iff 
\partial_t F^{-2} \begin{pmatrix}
\eta_{\rm k} \\ \xi_{\rm k}
\end{pmatrix} = 
\partial_t \begin{pmatrix}
\eta_{-{\rm k}} \\\xi_{-{\rm k}}
\end{pmatrix} = 
J \begin{pmatrix}
\partial_{\eta_{\rm k}} \widetilde H \\
\partial_{\xi_{\rm k}} \widetilde H
\end{pmatrix} \,,
\end{equation*}
where we have used that $\mathcal{F}^{-2} (\eta_{\rm k}, \xi_{\rm k}) = (\eta_{-{\rm k}}, \xi_{-{\rm k}})$.
The  water wave system in the Fourier space identifies to (\ref{ww-hamiltonian-equation-Fourier}).

\section{Poisson bracket calculations}

\subsection{Useful identities} \label{App-A}


\begin{lemma}
\label{lemma-poisson-first-two-lines}
\begin{equation}
\label{poisson-summation-1-bis}
\begin{aligned}
& i\left\{\int F_{123} z_1 z_2 z_3 \delta_{123} d{\rm k}_{123} \,,  \int G_{456} 
\bar z_{-4} \bar z_{-5} \bar z_{-6} \delta_{456} d{\rm k}_{456} \right\}  \\[5pt]
&
  = -\int (F_{123}+F_{312}+F_{231}) 
(G_{456}+G_{645}+G_{564}) z_2 z_3 \bar z_{-5} \bar z_{-6} \delta_{123} 
\delta_{456} \delta_{14} d{\rm k}_{1\ldots6} \,,
\end{aligned}
\end{equation}
{\small{
\begin{equation}
\label{first-two-lines-delta-2}
\begin{aligned}
&\int F_{{\rm k}_1}^{{\rm k}_2,{\rm k}_3} G_{{\rm k}_4}^{{\rm k}_5,{\rm k}_6} z_2 z_3 \bar z_{-5} \bar 
z_{-6} \delta_{123} \delta_{456} \delta_{14} d{\rm k}_{1\ldots6}  
=  \int F_{-{\rm k}_1-{\rm k}_2}^{{\rm k}_1, {\rm k}_2} G_{-{\rm k}_3-{\rm k}_4}^{{\rm k}_3, {\rm k}_4} z_1 z_2 \bar z_{-3} 
\bar z_{-4} \delta_{1234} d{\rm k}_{1\ldots4} \,.
\end{aligned}
\end{equation}
}}
Assuming that  $F_{123} = F_{213}$ and $G_{456}=G_{546}$, then
\begin{equation}
\label{poisson-last-two-lines-bis}
\begin{aligned}
&i \left\{\int F_{123}  z_{1}  z_{2} \bar z_{-3} \delta_{123} d{\rm k}_{123} \,,  
\int G_{456} \bar z_{-4} \bar z_{-5} z_{6} \delta_{456} d{\rm k}_{456} 
\right\}\\[5pt]
= &  \int (-
4 F_{123} G_{456}
z_{2} \bar z_{-3} \bar z_{-5} z_{6} \delta_{123} \delta_{456} 
\delta_{14}     
+ 
 F_{123} G_{456} z_{1} z_{2} \bar z_{-4} \bar z_{-5} 
\delta_{123} \delta_{456} \delta_{36} ) d{\rm k}_{1\ldots6} \,,
\end{aligned}
\end{equation}
{\small{
\begin{equation}
\label{last-two-lines-positive-part}
\begin{aligned}
& \int F_{{\rm k}_1}^{{\rm k}_2, {\rm k}_3} G_{{\rm k}_4}^{ {\rm k}_5, {\rm k}_6} z_{2} \bar z_{-3} \bar z_{-5} 
z_{6} \delta_{123} \delta_{456} \delta_{14} d{\rm k}_{1\ldots6}
  = \int F_{-{\rm k}_1-{\rm k}_3}^{{\rm k}_1, {\rm k}_3} G_{-{\rm k}_2-{\rm k}_4}^{{\rm k}_4, {\rm k}_2} z_1 z_2 \bar 
z_{-3} \bar z_{-4} \delta_{1234} d{\rm k}_{1\ldots4} \,,
\end{aligned}
\end{equation}
\begin{equation}
\label{last-two-lines-negative-part}
\begin{aligned}
& \int F_{{\rm k}_1}^{{\rm k}_2, {\rm k}_3} G_{{\rm k}_4}^{{\rm k}_5, {\rm k}_6}  z_{1} z_{2} \bar z_{-4} \bar 
z_{-5} \delta_{123} \delta_{456} \delta_{36} d{\rm k}_{1\ldots6} 
  = \int F_{{\rm k}_1}^{{\rm k}_2, -{\rm k}_1-{\rm k}_2} G_{{\rm k}_3}^{{\rm k}_4, -{\rm k}_3-{\rm k}_4} z_1 z_2 \bar 
z_{-3} \bar z_{-4} \delta_{1234} d{\rm k}_{1\ldots4} \,.
\end{aligned}
\end{equation}
}}
\end{lemma}
We  only prove  (\ref{poisson-summation-1-bis}). The other identities are proved in a similar manner.
Applying the Poisson bracket formula (\ref{poisson-bracket-z}),
\begin{equation*}
\label{poisson-summation-2}
\begin{aligned}
&\left\{\int F_{123} z_1 z_2 z_3 \delta_{123} d{\rm k}_{123} \,,  \int G_{456} 
\bar z_{-4} \bar z_{-5} \bar z_{-6} \delta_{456} d{\rm k}_{456} \right\} \\[5pt]
& = - \frac{1}{i} \sum_{\substack{(l,m,n) \in P(1,2,3) \\ (p,q,r) \in 
P(4,5,6)}} \int F_{123}  G_{456} z_l z_m \bar z_{-p} \bar z_{-q} 
\delta_{123} \delta_{456} \delta_{14} d{\rm k}_{1\ldots6} \,,
\end{aligned}
\end{equation*}
where the summation goes over the sets of all permutations $P$ of  $(1,2,3)$ 
and $(4,5,6)$. 
We  then apply index rearrangements to turn all integrals in the 
summation into those with monomial $z_2 z_3 \bar z_{-5} \bar 
z_{-6}$. For example, rearranging $(1,2,3)\to (2,3,1)$ and $(4,5,6) \to 
(5,6,4)$, we get
$$
\int F_{123}  G_{456} z_1 z_2 \bar z_{-4} \bar z_{-5} \delta_{123} 
\delta_{456} \delta_{36} d{\rm k}_{1\ldots6} = \int F_{231}  G_{564} z_2 z_3 
\bar z_{-5} \bar z_{-6} \delta_{123} \delta_{456} \delta_{14} d{\rm k}_{1\ldots6} \,.
$$

\subsection{Proof of Proposition \ref{lemma-T2-coeff}}
\label{appendix-T2-coeff}
We look for terms of the form  $zz \bar z \bar z$  in the Poisson bracket $\{K^{(3)}, H^{(3)}\}$ with $H^{(3)}$ and $K^{(3)}$ given in 
 (\ref{H3-in-z-2})--(\ref{K3-in-z-2}). To distinguish between the indices associated  to $K^{(3)}$ and $H^{(3)}$,  we use $(1,2,3)$ for $K^{(3)}$ and  $(4,5,6)$  for $H^{(3)}$. We have
 $$
\begin{aligned}
i \, \{K^{(3)}, H^{(3)}\}_R = & \left\{ \int \frac{A_{123} ~z_1 z_2 z_3 }{\omega_1+\omega_2+\omega_3} \delta_{123} d{\rm k}_{123} \,, \int A_{456} \bar z_{-4} \bar z_{-5} \bar z_{-6} \delta_{456} d{\rm k}_{456} \right\}\\[5pt]
& - \left\{ \int \frac{A_{123} ~\bar z_{-1} \bar z_{-2} \bar z_{-3} }{\omega_1+\omega_2+\omega_3} \delta_{123} d{\rm k}_{123} \,, \int A_{456}  z_4 z_5 z_6 \delta_{456} d{\rm k}_{456} \right\}\\[5pt]
& + \left\{ \int \frac{A_{123} ~z_{1} z_{2} \bar z_{-3} }{\omega_1+\omega_2-\omega_3}  \delta_{123} d{\rm k}_{123} \,, \int A_{456} \bar z_{-4} \bar z_{-5} z_6 \delta_{456} d{\rm k}_{456} \right\}\\[5pt]
& - \left\{ \int \frac{A_{123} ~\bar z_{-1} \bar z_{-2} z_{3} }{\omega_1+\omega_2-\omega_3} \delta_{123} d{\rm k}_{123} \,, \int A_{456} z_{4} z_{5} \bar z_{-6} \delta_{456} d{\rm k}_{456} \right\} \,.
\end{aligned}
$$
The second and fourth lines can be  modified using the antisymmetry property of the Poisson bracket and interchanging the indices $(1,2,3)$ and $(4,5,6)$:
\begin{equation}
\label{K3-H3-poisson}
\begin{aligned}
i \, \{K^{(3)}, H^{(3)}\}_R = & \left\{ \int \frac{A_{123} ~z_1 z_2 z_3}{\omega_1+\omega_2+\omega_3}  \delta_{123} d{\rm k}_{123} \,, \int A_{456} \bar z_{-4} \bar z_{-5} \bar z_{-6} \delta_{456} d{\rm k}_{456} \right\}\\[5pt]
& + \left\{\int A_{123}  z_1 z_2 z_3 \delta_{123} d{\rm k}_{123} \,, \int \frac{A_{456} ~ \bar z_{-4} \bar z_{-5} \bar z_{-6} }{\omega_4+\omega_5+\omega_6}  \delta_{456} d{\rm k}_{456} \right\}\\[5pt]
& + \left\{ \int \frac{A_{123} ~ z_{1} z_{2} \bar z_{-3} }{\omega_1+\omega_2-\omega_3} \delta_{123} d{\rm k}_{123} \,, \int A_{456} \bar z_{-4} \bar z_{-5} z_6 \delta_{456} d{\rm k}_{456} \right\}\\[5pt]
& + \left\{ \int A_{123}  z_{1} z_{2} \bar z_{-3} \delta_{123} d{\rm k}_{123} \,, \int \frac{A_{456}  \bar z_{-4} \bar z_{-5} z_6 }{\omega_4 + \omega_5 - \omega_6}\delta_{456} d{\rm k}_{456} \right\}\\[5pt]
&:= i \, (R_1+R_2+R_3+R_4) \,, \\[5pt]
\end{aligned}
\end{equation} 
where we denote each line of (\ref{K3-H3-poisson}) by $R_1$, $R_2$, $R_3$, $R_4$ respectively.

\noindent
{\bf Step 1.}  
We show that the coefficient ${\rm I} = R_1+ R_2$.  %
Using  the   identity  (\ref{poisson-summation-1-bis}), 
 we get
\begin{equation}
\label{first-line-poisson-bracket}
\begin{aligned}
R_1& = \int \frac{A_{123}+A_{312}+A_{231}}{
\omega_1 + \omega_2 + \omega_3} (A_{456}+A_{645}+A_{564}) z_2 z_3 \bar z_{-5} \bar z_{-6} \delta_{123} \delta_{456} \delta_{14} d{\rm k}_{123456} \,.
\end{aligned}
\end{equation}
From (\ref{A-123}), 
$$
S_{123} = 
 \sqrt[4]{g |{\rm k}_1| |{\rm k}_2| |{\rm k}_3|}  ~\ell_{{\rm k}_1}^{{\rm k}_3} \,, \quad  
A_{123}+A_{312}+A_{231} = \frac{ S_{123}}{8 \pi \sqrt{2}} ( \ell_{{\rm k}_1}^{{\rm k}_3} + \ell_{{\rm k}_1}^{{\rm k}_2}  + \ell_{{\rm k}_2}^{{\rm k}_3} ) \,.
$$
To  simplify the integral in (\ref{first-line-poisson-bracket}), we use the identity (\ref{first-two-lines-delta-2}) with
$$
\begin{aligned}
 F_{{\rm k}_1, {\rm k}_2, {\rm k}_3} := \frac{A_{123}+A_{312}+A_{231}}{\omega_1 + \omega_2 + \omega_3} \,, \quad 
G_{{\rm k}_4, {\rm k}_5, {\rm k}_6} := A_{456}+A_{645} + A_{564} \,,
\end{aligned}
$$
and obtain
\begin{equation*}
\label{first-line-result}
\begin{aligned}
&R_1 = \int {\rm I}_A z_1 z_2 \bar z_{-3} \bar z_{-4} \delta_{1234} d{\rm k}_{1234} \,,
\end{aligned}
\end{equation*}
$$
\begin{aligned}
{\rm I}_A  = &~ \frac{g^{1/4}}{128 \pi^2} \frac{\sqrt[4]{|{\rm k}_1||{\rm k}_2||{\rm k}_3||{\rm k}_4| |{\rm k}_1+{\rm k}_2| |{\rm k}_3+{\rm k}_4|}}{\omega_{{\rm k}_1}+ \omega_{{\rm k}_2}+ \omega_{{\rm k}_1+{\rm k}_2}} \\
& \times (\ell_{{\rm k}_1}^{{\rm k}_2} + \ell_{{\rm k}_1+{\rm k}_2}^{ -{\rm k}_1} + \ell_{{\rm k}_1+{\rm k}_2}^{-{\rm k}_2})  (\ell_{{\rm k}_3}^{{\rm k}_4} + \ell_{{\rm k}_3+{\rm k}_4}^{ -{\rm k}_3} + \ell_{{\rm k}_3+{\rm k}_4}^{-{\rm k}_4}) \,.
\end{aligned}
$$
We repeat these steps to  compute $R_2$  (the second line of (\ref{K3-H3-poisson})) and find
\begin{equation*}
\label{second-line-result}
\begin{aligned}
&R_2 = \int {\rm I}_B z_1 z_2 \bar z_{-3} \bar z_{-4} \delta_{1234} d{\rm k}_{1234} \,,
\end{aligned}
\end{equation*}
$$
\begin{aligned}
{\rm I}_B  = &~ \frac{g^{1/4}}{128 \pi^2} \frac{\sqrt[4]{|{\rm k}_1||{\rm k}_2||{\rm k}_3||{\rm k}_4| |{\rm k}_1+{\rm k}_2| |{\rm k}_3+{\rm k}_4|}}{\omega_{{\rm k}_3}+ \omega_{{\rm k}_4}+ \omega_{{\rm k}_3+{\rm k}_4}} \\
& \times (\ell_{{\rm k}_1}^{{\rm k}_2} + \ell_{{\rm k}_1+{\rm k}_2}^{-{\rm k}_1} + \ell_{{\rm k}_1+{\rm k}_2}^{ -{\rm k}_2}) 
(\ell_{{\rm k}_3}^{{\rm k}_4} + \ell_{{\rm k}_3+{\rm k}_4}^{ -{\rm k}_3} + \ell_{{\rm k}_3+{\rm k}_4}^{ -{\rm k}_4}) \,.
\end{aligned}
$$
Thus  $R_1+R_2 = \int {\rm I} \, z_1 z_2 \bar z_{-3} \bar z_{-4} \delta_{1234} d{\rm k}_{1234}$
where ${\rm I} = {\rm I}_A + {\rm I}_B$ is given in (\ref{I-term-new}).

{\bf Step 2.} We show that the combination of $R_3$ and $R_4$ identifies to ${\rm II} + {\rm III}$.
We apply the identity (\ref{poisson-last-two-lines-bis}) 
and get
\begin{equation}
\label{poisson-third-line-1}
\begin{aligned}
R_3 
= & \int  \frac{4 A_{123} A_{456}}{\omega_1+\omega_2-\omega_3}
z_{2} \bar z_{-3} \bar z_{-5} z_{6} \delta_{123} \delta_{456} \delta_{14} d{\rm k}_{123456}  \\
& 
 - \int  \frac{A_{123} A_{456}}{\omega_1+\omega_2-\omega_3} z_{1} z_{2} \bar z_{-4} \bar z_{-5} \delta_{123} \delta_{456} \delta_{36} d{\rm k}_{123456} \,.
\end{aligned}
\end{equation}
To further simplify the RHS, we  turn all its monomials into $z_1 z_2 \bar z_{-3} \bar z_{-4}$. This can be done by using the  identities 
(\ref{last-two-lines-positive-part})  and (\ref{last-two-lines-negative-part}).
We now  identify the terms on the RHS of (\ref{poisson-third-line-1}) 
\begin{equation*}
\begin{aligned}
R_3& = \int ({\rm II}_A + {\rm III}_A) z_1 z_2 \bar z_{-3} \bar z_{-4} \delta_{1234} d{\rm k}_{1234} \,,
\end{aligned}
\end{equation*}
where ${\rm II}_A$ comes from the first term in (\ref{poisson-third-line-1}),
$$
\begin{aligned}
{\rm II}_A  = &~ \frac{g^{1/4}}{32 \pi^2} \frac{\sqrt[4]{|{\rm k}_1||{\rm k}_2||{\rm k}_3||{\rm k}_4| |{\rm k}_1+{\rm k}_3| |{\rm k}_2+{\rm k}_4|}}{\omega_{{\rm k}_1}+ \omega_{{\rm k}_1+{\rm k}_3}- \omega_{{\rm k}_3}} \\
& \times (\ell_{{\rm k}_1}^{{\rm k}_3} + \ell_{{\rm k}_1+{\rm k}_3}^{-{\rm k}_3} - \ell_{{\rm k}_1+{\rm k}_3}^{-{\rm k}_1}) 
(\ell_{{\rm k}_4}^{{\rm k}_2} + \ell_{{\rm k}_2+{\rm k}_4}^{ -{\rm k}_2} - \ell_{{\rm k}_2+{\rm k}_4}^{ -{\rm k}_4}) \,,
\end{aligned}
$$
and ${\rm III}_A$  from the second term in (\ref{poisson-third-line-1}),
$$
\begin{aligned}
{\rm III}_A  = & - \frac{g^{1/4}}{128 \pi^2} \frac{\sqrt[4]{|{\rm k}_1||{\rm k}_2||{\rm k}_3||{\rm k}_4| |{\rm k}_1+{\rm k}_2| |{\rm k}_3+{\rm k}_4|}}{\omega_{{\rm k}_1}+ \omega_{{\rm k}_2}- \omega_{{\rm k}_1+{\rm k}_2}} \\
& \times (\ell_{{\rm k}_1+{\rm k}_2}^{-{\rm k}_1} + \ell_{{\rm k}_1+{\rm k}_2}^{-{\rm k}_2} - \ell_{{\rm k}_1}^{{\rm k}_2}) 
(\ell_{{\rm k}_3+{\rm k}_4}^{-{\rm k}_3} + \ell_{{\rm k}_3+{\rm k}_4}^{ -{\rm k}_4} - \ell_{{\rm k}_3}^{{\rm k}_4}) \, .
\end{aligned}
$$
We repeat these steps to compute  $R_4$  (fourth line in (\ref{K3-H3-poisson})) and obtain
\begin{equation*}
\begin{aligned}
R_4 & = \int ({\rm II}_B + {\rm III}_B) z_1 z_2 \bar z_{-3} \bar z_{-4} \delta_{1234} d{\rm k}_{1234} \,,
\end{aligned}
\end{equation*}
where 
$$
\begin{aligned}
{\rm II}_B  = &~ \frac{g^{1/4}}{32 \pi^2} \frac{\sqrt[4]{|{\rm k}_1||{\rm k}_2||{\rm k}_3||{\rm k}_4| |{\rm k}_1+{\rm k}_3| |{\rm k}_2+{\rm k}_4|}}{\omega_{{\rm k}_4}+ \omega_{{\rm k}_2+{\rm k}_4}- \omega_{{\rm k}_2}} \\
& \times (\ell_{{\rm k}_1}^{{\rm k}_3} + \ell_{{\rm k}_1+{\rm k}_3}^{-{\rm k}_3} - \ell_{{\rm k}_1+{\rm k}_3}^{-{\rm k}_1}) 
(\ell_{{\rm k}_4}^{{\rm k}_2} + \ell_{{\rm k}_2+{\rm k}_4}^{ -{\rm k}_2} - \ell_{{\rm k}_2+{\rm k}_4}^{ -{\rm k}_4}) \,,
\end{aligned}
$$
and ${\rm III}_B$ comes from the second term in (\ref{poisson-third-line-1}),
$$
\begin{aligned}
{\rm III}_B  = & - \frac{g^{1/4}}{128 \pi^2} \frac{\sqrt[4]{|{\rm k}_1||{\rm k}_2||{\rm k}_3||{\rm k}_4| |{\rm k}_1+{\rm k}_2| |{\rm k}_3+{\rm k}_4|}}{\omega_{{\rm k}_3}+ \omega_{{\rm k}_4}- \omega_{{\rm k}_3+{\rm k}_4}} \\
& \times (\ell_{{\rm k}_1+{\rm k}_2}^{-{\rm k}_1} + \ell_{{\rm k}_1+{\rm k}_2}^{-{\rm k}_2} - \ell_{{\rm k}_1}^{{\rm k}_2}) 
(\ell_{{\rm k}_3+{\rm k}_4}^{-{\rm k}_3} + \ell_{{\rm k}_3+{\rm k}_4}^{ -{\rm k}_4} - \ell_{{\rm k}_3}^{{\rm k}_4}) \,.
\end{aligned}
$$
Thus
$
R_3+R_4 = \int({\rm II}_A+{\rm II}_B+{\rm III}_A+{\rm III}_B) z_1 z_2 \bar z_{-3} \bar z_{-4} \delta_{1234} d{\rm k}_{1234}
$
where ${\rm II}_A+{\rm II}_B$ identifies to ${\rm II}$ in (\ref{II-term-new}) and 
${\rm III}_A+{\rm III}_B$ identifies to ${\rm III}$ in (\ref{III-term}).
This completes the proof.

\section*{Acknowledgments}

A. K. thanks the Fields Institute for its
support and hospitality during the Fall 2020.
C. S. is partially supported by the NSERC (grant
number 2018-04536) and a Killam Research Fellowship from the Canada Council for the Arts.

\bibliographystyle{siamplain}

\end{document}